\documentclass[11pt]{amsart}
\usepackage{latexsym}
\usepackage{amssymb}
\usepackage{amsfonts}
\usepackage{graphicx}
\usepackage{epsf}
\usepackage[all]{xy}
\usepackage{delarray}
\usepackage{setspace}
\newtheorem{theorem}{Theorem}[section]

\newtheorem{lemma}[theorem]{Lemma}
\newtheorem{proposition}[theorem]{Proposition}
\newtheorem{definition}[theorem]{Definition}

\newtheorem{remark}[theorem]{Remark}
\newtheorem{conjecture}[theorem]{Conjecture}
\newtheorem{question}[theorem]{Question}

\newcommand{\Hom}{{\rm Hom}}

\newcommand{\End}{{\rm End}}
\newcommand{\Ext}{{\rm Ext}}

\newcommand{\ann}{{\rm ann}}

\newcommand{\dm}{{\rm\underline{dim}}}

\newcommand{\cf}{{\rm cf}\,}
\newcommand{\dv}{\underline{\rm dim}\,}

\newcommand{\Aa}{\mathcal{A}}
\newcommand{\Bb}{\mathcal{B}}
\newcommand{\Cc}{\mathcal{C}}

\newcommand{\Ff}{\mathcal{F}}

\newcommand{\Mm}{\mathcal{M}}

\def\NN{{\mathbb N}}

\def\KK{{\mathbb K}}
\def\ZZ{{\mathbb Z}}

\def\cdn{{\bf cdn}}

\begin{document}

\sloppy

\title[Tame-Wild Dichotomy]{The tame-wild dichotomy conjecture for infinite dimensional algebras}



\begin{abstract}
We prove the tame-wild dichotomy conjecture, due to D. Simson, for infinite dimensional algebras and coalgebras. The key part of the approach is proving new representation theoretic characterizations local finiteness. Among other, we show that the Ext quiver of the category ${\rm f.d.-}A$ of finite dimensional representations of an arbitrary algebra $A$ is locally finite (i.e. $\dim(\Ext^1(S,T))<\infty$ for all simple finite dimensional $A$-modules $S,T$) if and only if for every dimension vector $\underline{d}$, the representations of $A$ of dimension vector $\underline{d}$ are all contained in a finite subcategory (a category of modules over a finite dimensional quotient algebra). This allows one reduce the tame/wild problem to the finite dimensional case and Drozd's classical result. We prove a local-global principle for tame/wild (in the sense of non-commutative localization): 
a category of comodules is tame/not wild if and only if every ``finite" localization is so.  We give the relations to Simson's f.c.tame/f.c.wild dichotomy and also recover some other known results, and use the methods and various embeddings we obtain to give connections to other problems in the literature. We list several questions that naturally arise. 
\end{abstract}


\author{M.C. Iovanov}
\address{The University of Iowa\\
12 MacLean Hall\\
Iowa City, IA 52242, USA}
\address{e-mail: miodrag-iovanov@uiowa.edu, yovanov@gmail.com} 

\date{}

\maketitle


\section{Introduction and Preliminaries}

A landmark result that lays at the foundation of the representation theory of finite dimensional algebras is Drozd's famous tame-wild dichotomy. While the study of representations of finite dimensional algebras sprung out from representation theory of finite groups in positive characteristic, and was motivated by the Brauer-Thrall conjectures, it quickly became an individual field of study, and Drozd's result can be considered to have given some direction: it shows for what kind of algebras can one hope to achieve a complete classification of representations. Stated loosely, it says that any finite dimensional algebra is either tame - which means that its representations can be completely classified by finitely many 1-parameter families, or it is wild, in which case, the representation theory of {\it every} other finite dimensional algebra can be ``found" inside that of $A$. Moreover, these two classes of algebras are mutually exclusive. Among tame algebras, the algebras of finite type - those algebras which have only finitely many isomorphism types of indecomposable representations - form a special class. These algebras have received a lot of attention over time, starting with Gabriel's now classical ADE classification result for quivers of finite type. 

\vspace{.2cm}

In the representation theory of finite dimensional tame algebras, and of algebras of finite type, often infinite dimensional algebras appear as well. The polynomial algebra $\KK[X]$ over an algebraically closed field $\KK$ is one such example; its finite dimensional representations are described by nothing else but the Jordan canonical form of a matrix, and they thus form such 1-parameter families. More generally, among others, the extended Dynkin quivers (Euclidean quivers) can be infinite dimensional, and they are also of tame type. Hence, a very natural question arises: does the tame-wild dichotomy hold for infinite dimensional algebras? This question can be reinterpreted most naturally in the language of comodules over coalgebras, where it was raised by Simson. More specifically, by an observation of Takeuchy \cite{Tak}, if $\Cc$ is a $\KK$-linear abelian category, generated by objects of finite length and whose endomorphism rings of simples are finite dimensional, then $\Cc$ is equivalent to the category of comodules over a coalgebra, or of rational modules over a pseudocompact algebra \cite{G} (see also \cite{DNR}; such a category is called of finite type by Takeuchi). This applies to the categories ${\rm f.d.-}A$ of finite dimensional $A$-modules, and respectively ${\rm l.f.-}A$ of locally finite (=sum of finite dimensional submodules) $A$-modules over an arbitrary algebra $A$: these categories are equivalent to that of finite dimensional, and respectively, arbitrary, comodules over the finite dual coalgebra $A^0$. 
There are many examples of such categories of finite type, such as rational modules over algebraic groups or group schemes, locally finite modules over quantum groups, modules over compact groups, the category $\mathcal O$ associated to a Lie algebra, categories of (co)chain complexes of vector spaces or modules over finite dimensional algebras, categories of graded modules. 
A coalgebra is tame if in the category $\Mm=C{\rm-comod}$ of finite dimensional comodules over $C$, for every fixed dimension vector $\underline{d}$ of finite length, the isomorphism types of objects (comodules) of dimension vector $\underline{d}$ are parametrized by a finite set of 1-parameter families, except for possibly finitely many such isomorphism types. A coalgebra is wild if $\Mm$ contains the representation theory of every finite dimensional algebra, in the sense that for every finite dimensional (wild) algebra $A$ there exists a representation embedding (faithful exact functor which preserves indecomposables and respects isomorphism types) of the category $A{-\rm mod}$ of finite dimensional $A$-modules into $\Mm$. One of the main questions in coalgebras, representations of infinite dimensional algebras and representation type, is the above tame/wild question which is known as

\begin{conjecture}[Simson's Infinite Tame-Wild dichotomy conjecture]
Every coalgebra $C$ is either tame or wild, but not both.
\end{conjecture}   

As noted before, this conjecture contains the question about the tame wild dichotomy for infinite dimensional algebras as a particular case. There has been a lot of work done on this conjecture or motivated by it (see \cite{JMN, KS,L1,L2,SS,Sh,Sh2}; \cite{si1}-\cite{Si13}), as well as work on other related various tame-wild dichotomy problems (see e.g. \cite{DG,Ri1,Ri2,RS1,RS2}). A weak version of this tame-wild dichotomy (the ``half" infinite tame-wild dichotomy) is given in \cite{si7} (see also \cite{si4}) where it is shown that a tame coalgebra is not wild. The most general result to date on the infinite tame-wild dichotomy is obtained by Simson in  \cite{si11}; it is proved there that the tame-wild dichotomy holds for one-sided semiperfect coalgebras, as well as for the so called Hom-computable coalgebras (see Sections \ref{s.main} and \ref{s.fc} for precise definitions), thus extending the classical tame-wild dichotomy from finite dimensional algebras to an important class of coalgebras and categories of finite type. To achieve this, an interesting approach is employed in \cite{si11}: the keen observation there is that the problem can be re-interpreted,  for these classes of coalgebras, as a problem about certain subcategories of comodules called finitely copresented, which are equivalent in those cases to categories of modules over finite dimensional algebras, at which point Drozd's result can be applied. This reduction \cite{si11} is nevertheless not at all straightforward, but a consequence of deep technical work.

\vspace{.2cm}

Our first main result is to prove this conjecture, and hence the tame-wild dichotomy for infinite dimensional algebras, in its full generality. This is done in Section \ref{s.main}. Our method is different from those in \cite{si11}. It yields a self-contained approach to the full conjecture: we show that, in general, the ``tame" and ``not wild" properties are local in the sense that, for any coalgebra, they can be tested for finite dimensional subcoalgebras. For this, we use the notion of coefficient coalgebra. More specifically, the strategy is as follows: when showing that the ``not wild" property is local (in the above sense), one notices that for any representation embedding $F:A{\rm-mod}\longrightarrow \Mm$, where $\Mm=C{\rm-comod}$ is the category of finite dimensional $C$ comodules, the ``image" of $F$ must lie inside the subcategory of $D$-comodules, where $D$ is the coefficient coalgebra of $F(A)$; this $D$ is then finite dimensional since $F(A)$ is so. Such ideas are present and used already in \cite{si7} proving the weak (one half) tame-wild dichotomy; we use similar techniques here as part of our general local-global principle for tame and wild. The main novelty is the fact that the above mentioned ``locality" holds for tameness; for this, a key fact that we show is that if every finite dimensional subcoalgebra $D$ of $C$ is tame, then the $C$-comodules of a fixed dimension vector $\underline{d}$ will all have their coefficient coalgebra contained in a certain finite dimensional subcoalgebra $H(\underline{d})$ of $C$, which depends only on $\underline{d}$. In fact, we observe a more general statement (Lemma \ref{l.key}), that actually gives new characterizations of another important concept, that of locally finite coalgebras. These coalgebras correspond to categories of finite type whose Ext quiver is locally finite, i.e. has finitely many arrows between any two vertices; Lemma \ref{l.key} states that $C$ is locally finite if and only if for each dimension vector $\underline{d}$, there is a finite dimensional subcoalgebra $H$ of $C$ such that all comodules of this dimension vector are realized over $H$ (i.e. belong to $H{\rm-comod}$, equivalently, their coefficient coalgebra is contained in this $H$).

\vspace{.2cm}

We also examine the ``not wild" and ``tame" notions vis-a-vis noncommutative localization. Localization in coalgebras, in the sense of general localization in Grothendieck categories introduced by Gabriel \cite{Ga}, has been studied by many authors \cite{CGT,CNO,JMN,JMN2,NT1,NT2}, and its relation to studying representation type of algebras is not new. For categories of finite type $\Mm=C{\rm-Comod}$ it takes an easily described form in terms of idempotents $e$ in the dual algebra $C^*$, thanks to work of \cite{CGT}; for each such idempotent $e\in C^*$ a natural coalgebra $eCe$ is formed (analogous to the corner ring $eC^*e$) and the localization is then described as a functor $(-)e:C{\rm-Comod}\longrightarrow eCe{\rm-Comod}$. In Section \ref{s.tw}, we obtain our next main result: we show that, in fact, the notions of ``not wild" and ``tame" are also local in the sense of localization: that is, a coalgebra is tame/not wild if and only if every localization at a {\it finite} idempotent $e$, is tame/not wild; here, finite idempotent means an idempotent where only finitely many simple objects remain in $eCe{\rm -Comod}$ after localization. In the non-commutative setting, such finite localizations are the natural counterpart of commutative localization at maximal ideals: in the commutative case, there are no extensions between different simple objects, and so, to obtain all such finite localizations, one essentially only needs to consider localizations after which only one simple module remains. Equivalently, in the commutative case, to understand modules of finite length one only needs to localize at one maximal at a time.\\
Results concerning such preservation of tame and wild properties under localization for coalgebras were obtained before, for example, in \cite{JMN} under some conditions. In our general equivalent characterizations of tame/wild properties via finite localizations, the first main result on the tame-wild dichotomy is used. Furthermore, we give the relation to Simson's fc-tame/fc-wild dichotomy in Section \ref{s.fc}. In fact, the method in \cite{si11} is also related to localization; we show that the functors and categories constructed there, are in fact ``co-localization", which behaves well exactly in the case of Hom-computable coalgebras.

\vspace{.2cm}

We also give a few embedding results between the categories of {\it arbitrary} (co)modules of wild finite or infinite dimensional (co)algebras. We show that the category of locally finite modules over the free algebra in countably many variables $\KK\langle X_1,X_2,\dots,X_n,\dots\rangle$ embeds (in the sense of representation embedding), and then consequently, the category $C{\rm-Comod}$ for any conutable dimensional coalgebra, into the category of arbitrary modules of any fully wild finite dimensional algebra, and in particular, for any two fully wild coalgebras of at most countable dimension $C,D$ the categories $C{\rm-Comod}$ and $D{\rm-Comod}$ of all comodules over $C$ and $D$ respectively, can be representation embedded into each other. However, as far as embedding the category of finite dimensional comodules of a coalgebra, we show that given a coalgebra $C$, there is a representation embedding of $D{\rm-comod}$ into $C{\rm-comod}$ (between categories of finite dimensional comodules) for every countable dimensional coalgebra $D$ precisely when $C$ is not locally finite, a fact that gives yet another representation theoretic interpretation and characterization of the local finiteness of the Ext quiver. These results are contained in Section \ref{s.we} and Theorem \ref{t.main2}.

\vspace{.2cm}

We end by listing a number of questions regarding the f.c. tame-wild dichotomy  and relations to tame-wild in general, and  embeddings of categories of finite dimensional representations of and/or comodules over arbitrary quivers in Section \ref{s.bt3}, and examine also possible relations to another conjecture of D. Simson, which may be also referred to as a Brauer-Thrall 3 (BT3) conjecture \cite{Si12}

\vspace{.2cm}

We refer the reader for terminology on coalgebras to standard textbooks such as \cite{DNR,R0}, and as well as to \cite{ASS,ARS} for basics of the representation theory of finite dimensional algebras. For categories of modules or comodules we adopt the notations and conventions of \cite{si11}, which are standard to representation theory. With a general audience in mind, we recall most notions, terminology and results used throughout the paper. In this respect, using the previous observations on the algebra/coalgebra language, the reader interested in representations of algebras, can replace, in any statement, the category of all comodules, and respectively, finite dimensional comodules of a coalgebra $C$, by the category of locally finite modules, and respsectively, finite dimensional modules (representations) of an algebra $A$.



\section{The infinite tame-wild dichotomy}\label{s.main}

We fix the algebraically closed field $\KK$, and a linear category of finite type $\Mm$, so that $\Mm=C{\rm-Comod}$, the category of left comodules over some basic coalgebra $C$. By $C{\rm-comod}$ we denote the subcategory of objects of finite length in $M$, equivalently, of finite dimensional $C$-comodules. Let $S(i)_{i\in I}$ a system of representatives of the simple left $C$-comodules such that the coradical of $C$ is ${\rm soc}(C)=C_0=\bigoplus\limits_{i\in I}S(i)$. There is an indecomposable decomposition $C=\bigoplus\limits_{i\in I}E(i)$ of the left comodule $C$ (=the minimal injective cogenerator of $\Mm$), where each $E(i)$ is an injective hull of $S(i)$ contained in $C$. For any finite dimensional $C$-comodule $M$, let $\dm(M)\in\NN^{(I)}$ denote its dimension vector. We recall that if $A$ is an algebra, a $C$-$A$-bicomodule $L$ is a left $C$-comodule and right $A$-module satisfying the obvious compatibility condition that the comultiplication map is a morphism of $A$-modules, equivalently, $A$-multiplication is a morphism of comodules, equivalently, $L$ is a $C^*$-$A$-bimodule which is rational as a $C^*$-module. The following definition is classical for finite dimensional algebras, and is contained in \cite{si3,si4}. 

\begin{definition}\label{d.1}
(i) The coalgebra $C$ is called of tame representation type if the category $C{\rm-comod}$ of finite dimensional left $C$-comodules is of tame type, that is, for every dimension vector $v\in \NN^{(I)}$, there is an almost parametrizing family of $\KK[T]$-$C$-bicomodules $L_1,\dots,L_{r}$ (left $C$-comodules and right $\KK[T]$-modules such that comultiplication is colinear, equivalently, the $L_i$ are $\KK[T]$-$(C^*)^{{\rm op}}$-bimodules), which are finitely generated free as $\KK[T]$-modules (with $r=r_v$ depending on $v$), meaning that all but possibly finitely many left $C$-comodules $M$ of dimension vector $\dm(M)=v$ are of the form 
$$M\cong L_s \otimes_{\KK[T]} \frac{\KK[T]}{(T-\lambda)}$$
for some $1\leq s\leq r=r_v$ and some $\lambda\in \KK$.\\
(ii) The coalgebra $C$ is called of wild (representation) type (or simply wild) if there is a representation embedding $F:{\rm mod-} \Gamma_3(\KK)\longrightarrow C{\rm-comod}$ of the category o finite dimensional modules over the algebra $\Gamma_3(\KK)=\left(\begin{array}{cc}\KK & \KK^3 \\ 0 & \KK \end{array}\right)$ into $C{\rm-comod}$ (here,  by representation embedding we mean an exact functor $F$ which preserves indecomposables and  respects isomorphism type, i.e. $F(X)\cong F(Y)$ implies $X\cong Y$; see \cite{ASS,SS}).
\end{definition}




To better explain our main result, we introduce the following notion of local property: this will be a property of a coalgebra (or a category of finite type $\Mm$) which can be checked by reducing to finite dimensional subcoalgebras (or finite subcategories of $\Mm$).

\begin{definition}
A property ${\mathcal P}$, which applies to coalgebras, is said to be a local property if given a coalgebra $C$, then $C$ has property ${\mathcal P}$ if and only if every finite dimensional subcoalgebra $H$ of $C$ has property ${\mathcal P}$.
\end{definition}

For example, the property of being cosemisimple is a local property \cite[Chapter 3]{DNR}; that is, semisimplicity of $C{\rm-comod}$ is equivalent to semisimplicity of any subcategory $H{\rm-comod}$ for finite dimensional $H\subseteq C$. Similarly, if $n$ is a fixed positive integer, having Loewy length at most $n$ is a local property, and so is being serial as a coalgebra (i.e. category of finite dimensional comodules is serial; for basic definitions, see \cite{CGT2,LS,I3}). On the other hand, many (or rather, most) properties are not local: such are having finite Lowey length, being one-sided semiperfect (i.e. injective indecomposable comodules are finite dimensional), being locally finite, artinian, co-Noetherian (see \cite{CGT,CNO,GTNT,I5,WW} for definitions and properties), etc. Before we give the first result of interest here, we need to recall some standard coalgebra/representation theory terminology.

\subsection*{Coefficient coalgebra}

If $M$ is a left $C$-comodule with comultiplication map $\rho:M\rightarrow C\otimes M$, then there is a smallest subcoalgebra $H$ of $C$ such that $\rho(M)\subseteq H\otimes M$, and this coalgebra is called the {\it coefficient coalgebra} of $M$ and will be denoted by $\cf(H)$ (see \cite[Chapter 2]{DNR}). It is the dual of the notion of annihilator: indeed, regarded as a right module over the dual convolution algebra $C^*$, the annihilator of $M$ is ${ann}_{C^*}(M)=(\cf(M))^\perp\subseteq C^*$, where, as usual, for $X\subseteq C$, $X^\perp$ consists of the maps in $C^*$ which vanish on $X$  (and similarly, if $Y\subseteq C^*$, then $Y^\perp=\{x|\,f(x)=0, \,\forall \,f\in Y\}$). The representation theoretic connection is that if $\eta:A\rightarrow \End_\KK(M)$ is a finite dimensional representation of $A$, then the coefficients $\eta_{i,j}\in A^*$ (in some fixed basis of $M$) of this representation span a finite dimensional subcoalgebra of the representative coalgebra $A^0$ (also called finite dual coalgebra) of $A$. We will use the fact that if $M$ is finite dimensional, then $\cf(M)$ is finite dimensional too; moreover, if $0\rightarrow X\rightarrow Y\rightarrow Z\rightarrow 0$ is an exact sequence, then $\cf(X)\subseteq \cf(Y)$ and $\cf(Z)\subseteq \cf(Y)$. For a left (or right) coideal $V$ (subcomodule) of $C$ it is obvious that $V\subseteq \cf(V)$.

The following result is essentially contained \cite[Theorem 6.7(d)]{si7}; nevertheless, the proof there seems to be more involved, as the author remarks, also ``involving rather lengthy arguments used in the proof of" \cite[Theorem 6.10]{si3}. We provide here a simplified short argument based on the above concept.

\begin{lemma}\label{l.wildlocal}
The property ``not wild" is a local property.
\end{lemma}
\begin{proof}
This means that $C$ is not wild if and only if for every finite dimensional subcoalgebra $H$ of $C$, we have: ``$H$ is not wild". Obviously, if $H$ is a wild subcoalgebra of $C$, then $C$ is wild because of the full, faithful representation embedding $H{\rm-comod}\hookrightarrow C{\rm-comod}$ induced by the inclusion $H\subset C$ (the co-restriction of scalars; see \cite{DNR,Tak}). \\
Conversely, assume $C$ is wild; we show that there is some finite dimensional wild subcoalgebra $H$ of $C$. Let $F:{\rm mod-}\,W\longrightarrow C{\rm-comod}$ be an (exact) representation embedding with $W$ a wild finite dimensional algebra. Let $H=\cf(F(W))$; it is a finite dimensional subcoalgebra of $C$. If $N$ is any other finite dimensional $W$-module, then there is an epimorphism of right $W$-modules $W^n\rightarrow N\rightarrow 0$, and so, using exactness, $F(N)$ is a quotient of $F(W)^n$ ($F$ is necessarily also additive; \cite{ASS}). Hence, $\cf(F(N))\subseteq H=\cf(F(W)^n)=\cf(F(W))=H$, and therefore $N$ is an $H$-comodule. This means that the functor $F$ co-restricts to $F:{\rm mod-}W\longrightarrow H{\rm-comod}$, and since $H{\rm-comod}$ is a full exact subcategory of $C{\rm-comod}$, this shows that $F$ is a representation embedding and so $H$ is wild.
\end{proof}

We note that arguments of somewhat similar flavor appeared in \cite{si7} (see also \cite{si3}) where the so called weak tame-wild dichotomy is proved (see also Remark \ref{r.typo}). 


\subsection*{Tame as a local property}

In order to prove that ``tameness" is a local property, we need to recall a few facts and give a new characterization of locally finite coalgebras. We begin with the following remark which known for the most part.

\begin{remark}[The coefficient coalgebra]
If $M$ is a finite dimensional left $C$-comodule with comodule map $\rho:M\rightarrow C\otimes M$, and basis $\{x_i | i=1,\dots,n\}$, write $\rho(x_i)=\sum\limits_{j=1}^n c_{ij}\otimes x_j$; then $\cf(M)={\rm Span}\{c_{ij}|i,j\}$ for example by \cite[2.5.4]{DNR}. This shows easily that if $M$ is an arbitrary left comodule, $\cf(M)=\sum\limits_{X{\rm\,f.d.\,}\subseteq M}\cf(X)$ - the sum of over all finite dimensional subcomodules of $M$. Thus, if $\ann_{C^*}(M)$ is the anihilator of $M$ over the dual algebra $C^*$, we see that $\cf(M)^\perp=\ann_{C^*}(M)$. The inclusion $\subseteq$ is easy and the for converse, if $a\in C^*$ is such that $Ma=0$, then for any finite dimensional subcomodule $X$ of $M$, pick a basis $x_i$ as above, and note that $x_i\cdot a=\sum\limits_{j=1}^n a(c_{ij})x_j=0$ for all $i$ so $a(c_{ij})=0$ for all $i,j$; hence $a$ is $0$ on $\cf(X)$ for all finite dimensional $X\subseteq M$. Thus, $a\in\cf(M)^\perp$. Note that in \cite[Proposition 2.5.3]{DNR} only the ``weaker" statement $\cf(M)=\ann_{C^*}(M)^\perp$ is proved; one can also use this to prove that $\cf(M)^\perp=\ann_{C^*}(M)$ by arguing that $M$ is a rational $C^*$-module and so $\ann_{C^*}(M)$ is closed in the finite topology of $C^*$. 
\end{remark}

We recall that for subspaces $U,W$ of $C$, the space $V\wedge W$ is defined as $V\wedge W=\{x\in C| \Delta(x)\in V\otimes C+C\otimes W\}=\Delta^{-1}(V\otimes C+C\otimes W)$ (see \cite[Section 2.5]{DNR}, \cite[Chapter 5]{Sw} or \cite{R0}). We note the following representation theoretic meaning of the wedge, relevant to extensions, which we believe is ``morally" known. 

\begin{proposition}
Let $Y$ be a left $C$-comodule, and $U,W$ two subcoalgebras of $C$. There is an exact sequence $0\rightarrow X\rightarrow Y\rightarrow Z\rightarrow 0$ of left $C$-comodule with $\cf(X)\subseteq U$ and $\cf(Z)\subseteq W$ if and only if $\cf(Y)\subseteq W\wedge U$. In particular, $\cf(Y)\subseteq \cf(Z)\wedge \cf(X)$.\\
Similarly, if $Y$ is a right $C$-comodule, the existence of such a sequence is equivalent to $\cf(Y)\subseteq U\wedge W$; in particular, in that case, $\cf(Y)\subseteq \cf(X)\wedge\cf(Z)$.
\end{proposition}
\begin{proof}
If such a sequence exists, then $W^\perp\cdot U^\perp\subseteq \ann_{C^*}(Z)\ann_{C^*}(X)\subseteq \ann_{C^*}(Y)$ (everything is a right $C^*$-module), and taking orthogonals, we get $W\wedge U=(W^\perp \cdot U^\perp)^\perp\supseteq \ann_{C^*}(Y)^\perp=\cf(Y)$, where the first equality follows, for example, by \cite[Lemma 2.5.7]{DNR}.\\
Conversely, $\cf(Y)\subseteq W\wedge U$ means $\ann_{C^*}(Y)\supseteq (W\wedge U)^\perp =(W^\perp U^\perp)^\perp{}^\perp\supseteq W^\perp U^\perp$ (or, one can easily see directly that $W^\perp U^\perp\subseteq \cf(Y)^\perp=\ann_{C^*}(Y)$ by hypothesis and the definition of $W\wedge U$). Let $T_U$ be the pre-torsion functor associated with $U$ and let $X=T_U(Y)$; this means $X$ is the largest subcomodule of $Y$ whose coefficient coalgebra is contained in $U$, equivalently, the largest $C^*$-submodule of $Y$ annihilated by $U^\perp$ (see \cite{NT1} or \cite[Section 2.5]{DNR}). Let $Z=Y/X$; then $X\cdot U^\perp =0$ and $Y W^\perp \subseteq X$ by the definition of $X$, since $(YW^\perp) U^\perp=0$. This implies $(Y/X)W^\perp=0$. This means $\cf(X)\subseteq U$, $\cf(Z)\subseteq W$ and so $0\rightarrow X\rightarrow Y\rightarrow Z\rightarrow 0$ is the required sequence.
\end{proof}

We recall that a coalgebra is said to be {\it locally finite} (see \cite{HR}) if for every two finite dimensional vector subspaces $V,W$ of $C$, the space $V\wedge W$ is finite dimensional as well. A comodule $M$ is said to be {\it quasifinite} \cite{Tak} if $\Hom(S,M)$ is finite dimensional for every simple (equivalently, every finite dimensional) comodule $S$. This is equivalent to asking that the multiplicity of $S$ in $Soc(M)$ is finite. \cite[Lemma 1.2]{I2} gives a few equivalent characterizations of locally finite coalgebras; in particular, of interest is that $C$ is locally finite if and only if $U\wedge W$ is finite dimensional for every finite dimensional subcoalgebras $U,W$ of $C$, if and only if $\Ext^1_C(S,T)$ is finite dimensional for every two (left/right) simple (equivalently, finite dimensional) $C$-comodules $S,T$. We give another new more representation theoretic characterization which will be the key of our result. Recall that if $C$ is a pointed coalgebra, that is, $\dim(S(i))=1$ for all $i$, the dimension vector of a finite dimensional comodule $M$ is defined as $\dv(M)\in \NN^{(I)}$ given by letting $\dv(M)_i$ equal the multiplicity of $S(i)$ in a Jordan-H${\rm \ddot{o}}$lder series of $M$. By extension, this vector can defined for any comodule over a coalgebra $C$ over any field $\KK$ as a vector recording multiplicities of simples. For some vector $\underline{d}\in\NN^{(I)}$ we define the ``coefficients" coalgebra $\cf(\underline{d})$ of $\underline{d}$ to be the smallest subcoalgebra of $C$ over which every comodule of dimension vector $\underline{d}$ can be realized; that is
$$\cf(\underline{d})=\sum\limits_{\dv(M)=\underline{d}}\cf(M)$$

The key step is the next 

\begin{lemma}\label{l.key}
Let $C$ be a coalgebra. Then $C$ is locally finite if and only if $\cf(\underline{d})$ is finite dimensional for every $\underline{d}\in\NN^{(I)}$.
\end{lemma}
\begin{proof}
Assume first the latter condition holds, and we prove $C$ is locally finite. For every simple comodule $S=S(i)$, we show that $E(S)/S$ is quasifinite, equivalently, $C/S$ is quasifinite; this will imply that $C$ is locally finite by \cite[Lemma 1.2]{I2}. Let $T$ be another simple comodule and $\underline{d}$ the dimension vector of a length two $C$-comodule with $S,T$ as its subquotients. If $f:T\rightarrow E(S)/S$ is a non-zero morphism, pulling back we find a length two subcomodule $V$ of $E(S)$ with top $T$ and socle $S$, and so $\dv(V)=\underline{d}$. Hence, $\cf(V)\subseteq \cf(\underline{d})$ which is finite dimensional by hypothesis. Now, $E(S)\subseteq C$, and so $V\subseteq \cf(V)\subseteq \cf(\underline{d})$. This shows that ${\rm Im}(f)=V/S\subseteq \cf(V)/S\subseteq (E(S)\cap \cf(\underline{d}))/S$, so $f$ factors to $f:T\rightarrow (E(S)\cap \cf(\underline{d}))/S$. Hence, $\dim(\Hom^C(T,E(S)))= \dim\Hom^C(T,(E(S)\cap \cf(\underline{d}))/S)$ is finite since $(E(S)\cap \cf(\underline{d}))/S$ is finite dimensional ($\dim(\cf(\underline{d}))<\infty$), which ends the proof of this implication. \\
Conversely, assume that $C$ is locally finite. We prove the statement by induction on $|\underline{d}|$, where $|\underline{d}|$ is the sum of the entries in $\underline{d}$ (the length of $|\underline{d}|$). Consider the partial order on $\NN^{(I)}$ given by component-wise comparisons of entries ($\underline{d}\leq \underline{e}$ if $d_i\leq e_i$). First, note that for $|\underline{d}|=1$ the statement is obvious as in that case $\cf(\underline{d})$ is a simple subcoalgebra of $C$. Inductively, if the statement holds for $e\in\NN^{(I)}$ with $|\underline{e}|<n$, consider $|\underline{d}|=n$. Any finite dimensional $C$-comodule $Y$ of representation dimension $\underline{d}$ fits into an exact sequence $0\rightarrow X\rightarrow Y\rightarrow Z\rightarrow 0$ with $\underline{e}=\dv(X)<\dv(Y)=\underline{d}$ and $\underline{f}=\dv(Z)<\dv(Y)=\underline{d}$. By the previous proposition, $\cf(Y)\subseteq \cf(Z)\wedge \cf(X)\subseteq \cf(\underline{f})\wedge \cf(\underline{e})$. This shows that
$$\cf(\underline{d})=\sum\limits_{\dv(Y)=\underline{d}}\cf(Y)\subseteq \sum\limits_{\underline{e},\underline{f}\in \NN^{(I)};\,\,\, \underline{e},\underline{f}<\underline{d}}\cf(\underline{f})\wedge \cf(\underline{e}) $$
(the last sum is actually over all nontrivial decompositions $\underline{e}+\underline{f}=\underline{d}$). But by induction $\cf(\underline{f}),\cf(\underline{e})$ are finite dimensional when $\underline{e},\underline{f}<\underline{d}$, and by the locally finite hypothesis, their wedge is finite dimensional too. Since the sum is over finitely many $\underline{e},\underline{f}$'s, we get that $\cf(\underline{d})$ is finite dimensional, and the proof is finished.
\end{proof}

\begin{remark}\label{r.typo}
If $C$ is a tame coalgebra, then any subcoalgebra $D$ of $C$ is tame. This is analogous to the well known statement that a quotient of a tame algebra is tame, and is essentially known \cite[Theorem 6.7(e)]{si7} (there seems to be a small typo at the end of the proof). One only needs to observe that if $C$ is tame, $\underline{d}$ is a dimension vector over $D$ and $L_1,\dots,L_n$ is an almost parametrizing family for the $C$-comodules $M$ of $\dv(M)=\underline{d}$, then $S_i=L_i/(L_i\cdot D^\perp)=L_i\otimes_{C^*}C^*/D^\perp$ are a parametrizing family for $D$-comodules of dimension vector $\underline{d}$. Indeed, given any such $D$-comodule $N$ of $\dv(N)=\underline{d}$, it can be regarded as a $C$-comodule annihilated by $D^\perp$, and so, except for possibly finitely many isomorphism types of $D$-comodules $N$, we have $N\cong \KK[T]/(T-\lambda)\otimes_{\KK[T]}L_i$, as left $C$-comodules (and right $C^*$-modules) for some $\lambda$ and some $i$; 
we thus get $N\cong N/ND^\perp\cong \KK[T]/(T-\lambda)\otimes_{\KK[T]}L_i\otimes_{C^*}C^*/D^\perp\cong \KK[T]\otimes_{\KK[T]} S_i$.  We note that in the proof of \cite[Theorem 6.7(e)]{si7}, the $\KK[T]$-$C$-bicomodules $L_i$ are stated to be regarded as $\KK[T]$-$D$-bicomodules; however, there is no guarantee that their coefficient coalgebra (over $C$) is contained in the subcoalgebra $D$, and hence, one needs to apply the above mentioned functor, which is a natural step as this functor is the left adjoint to the co-restriction of scalars $D{\rm-Comod}\rightarrow C{\rm-Comod}$. 
\end{remark}

We can now deal with the tame property:

\begin{lemma}\label{l.tamelocal}
Tameness (of coalgebras) is a local property.
\end{lemma}
\begin{proof}
If $C$ is tame, then so is any (finite dimensional) subcoalgebra by the remark above.\\
Conversely, assume $C$ is locally tame, and we show first that $C$ is locally finite. In fact, $\dim(\Ext^1_C(S(i),S(j)))\leq 2$ for any pair of simple comodules $S(i),S(j)$. Indeed, otherwise, consider the embedding of $C$ in the quiver coalgebra $\KK Q$ of the Ext-quiver $Q$ of $C$; this embedding is such that $C$ contains all vertices and arrows in $Q$. If $\dim(\Ext^1_C(S(i),S(j)))\geq 3$ for some pair of vertices $S(i),S(j)$ of $Q$, then either the Kronecker quiver $\Gamma_3$ is contained in $Q$ if $S(i)\not\cong S(j)$, or the quiver $Q'$ with one vertex and three loops is contained in $Q$ if $S(i)\cong S(j)$. Hence, the quiver coalgebra of either $\Gamma_3$ or $Q'$ is contained in $C$, and this implies that $C$ contains either (a copy of) the quiver coalgebra of $\Gamma_3$, or (a copy of) the subcoalgebra $(\KK Q')_1$ of $\KK Q'$ spanned by the vertex and arrows of $Q'$. The comodule categories of these two coalgebras are equivalent to ${\rm mod-}\Gamma_3$ and ${\rm mod-}\KK[t_1,t_2,t_3]/(t_1,t_2,t_3)^2$, respectively, both being wild. This contradicts the locally tame hypothesis, that every finite dimensional subcoalgebra is tame (this uses the Drozd's tame-wild dichotomy for finite dimensional algebras). \\
To end the proof, we note the key novelty, which again uses the idea of coefficient coalgebra. Let $\underline{d}$ be a dimension vector over $C$; then $\cf(\underline{d})$ is finite dimensional since $C$ is locally finite. Let $L_1,\dots,L_n$ be an almost parametrizing family for $\cf(\underline{d})$-comodules of dimension vector $\underline{d}$; the $L_i$'s are $\KK[T]$-$\cf(\underline{d})$-bicomodules, but they are also $\KK[T]$-$C$-bicomodules by corestriction. Then any $C$-comodule $M$ of $\dv(M)=\underline{d}$ is a $\cf(\underline{d})$-comodule, since $\cf(M)\subseteq \cf(\underline{d})$, and so, except for possibly finitely many such $M$'s, there is an isomorphism $M\cong \KK[T]/(T-\lambda)\otimes_{\KK[T]}L_i$ of $\cf(\underline{d})$-comodules for $\lambda\in \KK$, which is in fact an isomorphism of $C$-comodules by corestriction. The Lemma is proved. 
\end{proof}

Now, the main result follows immediately from Drozd's tame-wild dichotomy for finite dimensional algebras.

\begin{theorem}
Any coalgebra $C$ over an algebraically closed field is either tame or wild, and not both.
\end{theorem}
\begin{proof}
Both ``tame" and ``not wild" are local properties; since they are equivalent for every finite dimensional coalgebra (the category of (finite dimensional) comodules over a finite dimensional algebra is equivalent to that of (finite dimensional) modules over the dual algebra), they are equivalent locally, and so also globally.
\end{proof}


\section{Wild Embeddings}\label{s.we}

In this section, we exhibit some representation embeddings between locally finite modules over wild algebras. This is somewhat more relaxed than the usual representation embeddings, which take place between categories of finite dimensional representations. We will denote ${\rm lf-}A$ the category of locally finite right $A$-modules over an algebra $A$. Recall that an $A$-module is locally finite if it is the sum of its finite dimensional submodules. We start with the following proposition, which shows that in this setting one can embed representations over free algebras of uncountably many variables.

\begin{proposition}\label{p.NN}
Let $\KK$ be an infinite field. There is a full and faithful (exact) representation embedding of the category of locally finite modules ${\rm lf-}\KK\langle x_n| n\in\NN\rangle$ over the noncommutative polynomial algebra $A=\KK\langle x_n| n\in\NN\rangle$ in countably many variables into the category ${\rm lf-}\KK\langle y,z,t\rangle$ of locally finite modules over the noncommutative polynomial algebra in three variables.
\end{proposition}
\begin{proof}
For $M\in {\rm lf-}\KK\langle x_n| n\in\NN\rangle$, define $F(M)=M^{(\NN)}$, and let $F(M)_n$ denote the corresponding component of this direct sum. For $w\in M$, let $w_{(n)}$ denote the element of $F(M)$ which is the image of $w$ via the canonical injection $M\cong F(M)_n\hookrightarrow F(M)$.  Let $y$ act as $\bigoplus\limits_n (w_{(n)}\mapsto w_{(n)}\cdot x_n)$ on $F(M)=M^{(\NN)}$ (multiplication by $x_n$ in each component) and $z$ act on the right as a shift, so $z:F(M)_n\rightarrow F(M)_{n-1}$ is defined by $w_{(n)}\cdot z=w_{(n-1)}$ if $n\geq 1$ and let $z$ act as $0$ on $F(M)_0$. Let also $t$ act diagonally as $\bigoplus\limits_{n}\lambda_n$, where $\lambda_n$ are (pariwise) distinct elements of $\KK$. This defines a $\KK\langle y,z,t\rangle$-module structure on $F(M)$. If $\varphi:M\rightarrow N$ is a morphism of $A$-modules, define $F(\varphi)=\varphi^{(\NN)}$. It is straightforward to note that $F$ is a functor. Moreover, if $w_{(n)}\in F(M)_n$, then it is easy to check that $w_{(n)}\cdot f(y,z,t)\in \bigoplus\limits_{i\leq n} w_{(i)}\cdot A$; this is finite dimensional since $w\cdot A\subseteq M$ is finite dimensional (this is because the action of $z$ shifts down, and the action of $y$ is the action of one of the elements $x_n\in A$, and $t$ acts diagonally). Thus, $F(M)$ is locally finite.\\
It is obvious that $F$ is faithful. Let $\theta: F(M)\rightarrow F(N)$ be a morphism of $\KK\langle y,z,t\rangle$-modules. If $w_{(n)}\in F(M)_{n}$, write $\theta(w_{(n)})=\sum\limits_k\theta(w_{(n)})_k$ with respect to the direct sum components in $F(N)^{(\NN)}$. Since $\theta(w_{(n)} \cdot t)=\theta(w_{(n)}) \cdot t$, we get
$$\lambda_n \sum\limits_{k}\theta(w_{(n)})_k= \sum\limits_{k}\lambda_k\cdot \theta(w_{(n)})_k$$
Identifying component-wise, this means $\theta(w_{(n)})_k=0$ if $k\neq n$ so $\theta(F(M)_n)\subseteq F(N)_n$. Thus, $\theta=\bigoplus\limits_n \theta_n$ with $\theta_n:M\rightarrow N$, and using  
 that $\theta$ commutes with $z$ it will easily follow that $\theta_n=\theta_{n-1}$ for all $n\geq 1$, and so, in fact, $\theta=\varphi^{(\NN)}$. Finally, the fact that $\theta$ commutes with $y$ readily translates to the fact that each for each $n$, $\theta_n=\varphi$ commutes with $x_n$. Hence, $\theta=\varphi^{(\NN)}$ for $\varphi$ a morphism of $A$-modules, and so $F$ is full. \\
It is easy to see that $F$ is also exact, and as $F$ is full and faithful, it follows that it is also a representation embedding. This ends the proof.
\end{proof}

We now use several embeddings to relate all countably generated algebras; these are simply the embeddings used classically to relate finite dimensional representations, but now they will be extended to locally finite modules. The following natural lemma will be useful for this.\\


\begin{lemma}\label{l.fininf}
Let $F:C{-\rm Comod}\longrightarrow {\rm}D{-\rm Comod}$ be an exact functor, which restricts to a functor $F\vert:C{-\rm comod}\longrightarrow {\rm}D{-\rm comod}$. If the restriction $F\vert$ is full and faithful, then so is $F$.
\end{lemma}
\begin{proof}
To show $F$ is full, take $\alpha:F(X)\rightarrow F(Y)$, write $X=\lim\limits_{\stackrel{\longrightarrow}{i}} X_i$ with $X_i$ the finite dimensional submodules of $X$ and similarly $Y=\lim\limits_{\stackrel{\longrightarrow}{j}}Y_j$. By (right) exactness, we have that $F(X)=\lim\limits_{\stackrel{\longrightarrow}{i}}F(X_i)$, $F(Y)=\lim\limits_{\stackrel{\longrightarrow}{j}}F(Y_j)$. Then for each $X_i$, $\alpha(F(X_i))$ is a finite dimensional submodule of $F(Y)$ and is contained in some $F(Y_j)$. Re-arranging indices, we may pick (fix) for each $i$ some $Y_{j(i)}$ such that $\alpha(F(X_i))\subseteq F(Y_{j(i)})$. Since $F\vert$ is full, this means that  $\alpha_i=\alpha\vert_{F(X_i)}:F(X_i)\rightarrow F(Y_{j(i)})$ is given by $\alpha_i=F(f_i)$ for $f_i:X_i\rightarrow Y_{j(i)}$. Furthermore, the choices of $f_i's$ are unique since $F$ is faithfull on finite dimensional, and then the $f_i$ form a (directed) inductive system, again by the Faithfulness of $F\vert$ (since their images $\alpha_i$ do). Thus, $\alpha=\lim\limits_{\stackrel{\longrightarrow}{i}} \alpha_i=\lim\limits_{\stackrel{\longrightarrow}{i}} F(f_i)=F(\lim\limits_{\stackrel{\longrightarrow}{i}}f_i)$, which proves that $F$ is full. \\
To prove $F$ is faithful, let $f:X\rightarrow Y$ be a morphism of (arbitrary) $C$-comodules with $F(f)=0$. 
Write $f=\lim\limits_{\stackrel{\longrightarrow}{i}} f_i$ where $f_i:X_i\rightarrow Y_i$ are morphisms between finite dimensional submodules of $X$ and $Y$ respectively; moreover, $F(f)=\lim\limits_{\stackrel{\longrightarrow}{i}} F(f_i)=0$. Here, since $F$ is left exact, $F(X_i)$ can be regarded as submodules of $F(X)$, and therefore, $F(f_i)$ can be regarded as the restrictions of $F(f)$ to $F(X_i)$, and so they must be $0$. Since $F$ is faithful on finite dimensional objects, it follows that $f_i=0$ for all $i$, and this ends the proof. 
\end{proof}

We now consider the following embeddings.\\
(1) By \cite[pages 315-318, Theorem 1.7]{ASS3}, if $B$ is a finitely generated algebra, there is a $B$-$\KK\langle u,w\rangle$-bimodule $M$ which is finitely generated free as $B$-module, such that the functor $G=(-)\otimes_B M:{\rm Mod-}B\longrightarrow {\rm Mod-}\KK\langle u,w\rangle$ is full, faithful and exact, and takes finite dimensional modules to finite dimensional modules. Then, since the functor $F$ commutes with colimits, it takes locally finite modules to locally finite modules, and by the previous lemma, it remains full and faithful. \\
(2) Similarly, there is an additive functor $H:{\rm Mod-}\KK\langle u,w\rangle\longrightarrow {\rm Mod-}\Gamma_3(\KK)$ (the category of all representations of the $\Gamma_3$ quiver), which is exact, carries finite dimensional modules to finite dimensional modules, and it is furthermore is full and faithful when restricted to $H:{\rm mod-}\KK\langle u,w\rangle\longrightarrow {\rm mod-}\Gamma_3(\KK)$. Hence, since $H$ also exact, carries locally finite to locally finite modules (in fact, any module over $\Gamma_3(\KK)$ is locally finite, since this is a finite dimensional algebra). Again, the previous lemma shows that $H$ remains full and faithful on the full categories of locally finite modules.\\
(3) Finally, for every wild algebra $W$, there is such an (additive exact) representation embedding $R:{\rm mod-}\Gamma_3(\KK)\longrightarrow {\rm mod-}W$, which is given by a tensor product $R=(-)\otimes_{\Gamma_3(\KK)}N$ (\cite[1.6 Corollary, page 314]{ASS3}; $N$ is again finitely generated projective over $\Gamma_3(\KK)$), and thus this formula defines again a functor $I=(-)\otimes_{\Gamma_3(\KK)}N$ between the full categories of all (locally finite) modules too. Assume $W$ is {\it fully wild} (or {\it strictly wild}), which means that there is such a representation embedding $R$ which is full and faithful. In this case, again $R=I$ remains full and faithful when considered on locally finite modules. Composing these functors $I\circ H\circ G\circ F$, where $F$ is the functor of Proposition \ref{p.NN}, we obtain the following.


\begin{proposition}
If $\KK$ is an infinite field and $W$ is a fully wild algebra, then there is an exact additive fully faithful functor (representation embedding) $U:{\rm lf-}\KK\langle x_n|n\in \NN\rangle\longrightarrow {\rm lf-}W$, which is moreover given by $U=(-)\otimes_{\KK\langle x_n|n\in \NN\rangle} M$ for a bimodule $M$ which is projective to the left. In particular, when $W$ is finite dimensional, this yields an exact fully faithful embedding $U:{\rm lf-}\KK\langle x_n|n\in \NN\rangle\longrightarrow {\rm Mod-}W$.
\end{proposition}

We use this to show that the category of comodules of {\it any} coalgebra $C$ of at most countable dimension can be representation embedded in the category of finite dimensional comodules of any finite dimensional fully wild algebra  $W$, and consequently, into the category of comodules over any fully wild coalgebra $D$. 

To make the last connection, we will need the following consrtruction. As before, let $A=\KK\langle x_n|n\in \NN\rangle=\KK\langle\NN\rangle$. We first make the following observation. Denote ${\mathbf n}=\{1,2,\dots,n\}$, and let $\KK\langle {\mathbf n}\rangle=\KK\langle x_1,\dots,x_n\rangle$ be the corresponding noncommutative algebra of polynomials, and when $n<m$, we write $\pi_{n,m}:\KK\langle {\mathbf m}\rangle \rightarrow \KK\langle {\mathbf n}\rangle$ the projection corresponding to setting the variables in ${\mathbf m}\setminus {\mathbf n}$ to $0$, and mapping the other ones to their respective counterparts. If $A,B$ are finite dimensional algebras with a projection $\pi:A\rightarrow B$, $I=\ker(\pi)$, let $B$ be generated by $\pi(a_1),\dots,\pi(a_n)$ for $a_1,\dots,a_n\in A$ (for example, $\{a_1,\dots,a_n\}$ can be chosen to be a basis of $A$ modulo $I$), and let $a_{n+1},\dots,a_m$ be a basis of $I$. Consider the surjections: $\KK\langle {\mathbf n}\rangle \rightarrow B=A/I$, $x_i\mapsto {\overline a_i}=\pi(a_i)\in B=A/I$, and $\KK\langle {\mathbf m}\rangle \rightarrow A$, $x_i\mapsto \begin{cases}{a_i}\in A & {\rm if\,}i\leq n\\ a_i\in I\subseteq A & {\rm if\,}n<i\leq m\end{cases}$. The following diagram is commutative; the maps are the ones above; the map $\KK\langle \NN\rangle\rightarrow \KK\langle {\mathbf m}\rangle$ sends all variables $x_i$ for $i>m$ to $0$.
$$
\xymatrix{
\KK\langle\NN\rangle\ar@{->>}[r] & \KK\langle {\mathbf m}\rangle \ar@{->>}[r]^{\pi_{n,m}} \ar@{->>}[d] & \KK\langle {\mathbf n}\rangle \ar@{->>}[d]\\
& A \ar@{->>}[r]_{\pi} & B 
}
$$
We now pass to the finite dual coalgebras. Recall that if $H$ is an algebra, the finite dual coalgera $H^0$ consists of all representative functions $f\in H^*$, equivalently, functions $f$ whose kernel contains a cofinite ideal. $H^0$ can be regarded as the following limit of coalgebras (with natural maps provided by inclusions of cofinite ideals): $H^0=\lim\limits_{\stackrel{\longrightarrow}{I{\rm\,cofinite\,ideal\,in\,}H}}(H/I)^*$. The finite dual coalgebra $H^0$ has the property that the category of finite dimensional, respectively, locally finite right $H$-modules is equivalent to the category of finite dimensional, respectively, arbitrary left $H^0$-comodules. Now let $D\hookrightarrow E$ be any inclusion (embedding) of finite dimensional algebras, and consider the diagram above for the induced projection $A=E^*\rightarrow D^*=B$. By taking finite duals in the above diagram (applying the finite dual functor), and using that $D=(D^*)^0; E=(E^*)^0$, 
we obtain a commutative diagram of coalgebras, with the natural maps:
$$
\xymatrix{
\KK\langle {\mathbf n}\rangle^0 \ar@{^{(}->}[r] & \KK\langle {\mathbf m}\rangle^0  \ar@{^{(}->}[r] & \KK\langle\NN\rangle^0\\
D \ar@{^{(}->}[u] \ar@{^{(}->}[r]& E \ar@{^{(}->}[u] & 
& 
}
$$
Now let $C$ be an {\it arbitrary} (at most) countable dimensional coalgebra. Let $C=\bigcup\limits_{n\in\NN}E_n$ with $E_n\subseteq E_{n+1}$ and $E_n$ finite dimensional subcoalgebras of $C$ (this can be done by choosing a basis first and taking subcoalgebras generated by finite increasing parts of the basis). Apply the diagram above to obtain a commutative diagram with natural maps:
$$
\xymatrix{
\KK\langle {\mathbf n_k}\rangle^0 \ar@{^{(}->}[r] & \KK\langle {\mathbf n_{k+1}}\rangle^0  \ar@{^{(}->}[r] & \KK\langle\NN\rangle^0\\
E_k \ar@{^{(}->}[u] \ar@{^{(}->}[r]& E_{k+1} \ar@{^{(}->}[u]  \ar@{^{(}->}[r] & C \ar@{^{(}.>}[u]
& 
}
$$
Here, the $n_k$ will be a suitably chosen increasing sequence of natural numbers (depending on what is needed to generate $E_k^*$ and $E_{k+1}^*$). Taking colimits, we obtain an embedding $C=\lim\limits_{\stackrel{\longrightarrow}{k}}E_k\hookrightarrow \lim\limits_{\stackrel{\longrightarrow}{k}}\KK\langle {\mathbf n_k}\rangle^0=\lim\limits_{\stackrel{\longrightarrow}{n}}\KK\langle {\mathbf n}\rangle^0\hookrightarrow \KK\langle\NN\rangle^0$. Hence, any countably dimensional coalgebra $C$ can be embedded into the coalgebra $\KK_\NN= \lim\limits_{\stackrel{\longrightarrow}{n}}\KK\langle {\mathbf n}\rangle^0$ and into $\KK\langle\NN\rangle^0$ (in fact, this can also be obtained using the cofree coalgebra \cite{DNR,R0}; see also \cite{AI} for a similar construction involving limits). The (co)restriction functor $C{\rm-Comod}\longrightarrow {\rm lf-}\KK\langle\NN \rangle=\KK\langle\NN \rangle^0{\rm-Comod}$ is exact, full and faithful, and so preserves indecomposables and reflects isomorphisms. Hence, combining all these results, we obtain

\begin{theorem}\label{t.wildE}
If $C,D$ are wild coalgebras of countable dimension and $D$ is fully wild, then there exists a full representation embedding $C{\rm-Comod}\longrightarrow D{\rm-Comod}$.
\end{theorem}

We remark that embeddings of the full module categories are considered in literature (see e.g. \cite{Ri1, Ri2, Sh2}). An algebra $A$ for which there is a (full) representation embedding ${\rm Mod-}W\longrightarrow {\rm Mod-}A$ for $W=\Gamma_3(\KK)$ (equivalently, any wild algebra $W$) is said to be (fully) Wild (see \cite[XIX]{ASS}). A fully wild algebra is fully Wild by a result of \cite{Si13}; thus, the previous theorem extends this in the infinite dimensional case, but gives a little more, that in fact the category of locally finite modules over any countably generated algebra, and hence, comodules and locally finite modules over arbitrary quivers, can be embedded into ${\rm Mod-}A$.
Of course, infinite dimensional algebras/coalgebras of at most countable dimension present particular interest. One may wonder whether the {\it finite dimensional} comodules over any wild countable dimensional coalgebra can be embedded into $W$-mod for a finite dimensional algebra as well, a ``stronger" (or more wild) form of embedding. The above inclusion $C\hookrightarrow \KK_\NN$ of any countable dimensional $C$ into the (cofree coalgebra) $\KK_\NN$ shows that, in fact, every category of finite dimensional comodules can be embedded into the category of finite dimensional $\KK_\NN$-comodules (and also of locally finite $\KK\langle\NN\rangle$-modules). The following shows that the above ``stronger" embedding between categories of finite dimensional objects cannot hold in general, not even if the coalgebra $C$ into which we attempt to embed is infinite dimensional, wild and connected. In fact, the situation when such embeddings don't exist turns out to characterize precisely locally finite coalgebras. 

\begin{proposition}\label{p.KronekerEmbed}
Let $Q=\Gamma_\infty$ be the ``infinite Kronecker" quiver  
$\xymatrix{ \bullet \ar@/^/[r] \ar[r] \ar@/_/[r]_{\stackrel{}{\dots}} & \bullet}$ with two vertices $a,b$ and infinitely many arrows $a\rightarrow b$. Then for any locally finite coalgebra $C$ (in particular, a finite dimensional (co)algebra $C$), the category ${\bf rep}_Q$ of finite dimensional $Q$ representations cannot be representation embedded into $C{\rm-comod}$.
\end{proposition}
\begin{proof}
Assume such a representation embedding $F:{\bf rep}_Q\longrightarrow C{\rm-comod}$ exists. Let $S_a,S_b$ be the simple left comodules corresponding to the vertices $a,b$; obviously $\Ext^1(S_a,S_b)$ is an infinite dimensional space. We note that $F$ is linear functor (which follows from the definition of representation embedding). Note that if $\zeta:\,\,\,0\rightarrow S_b\rightarrow M\rightarrow S_a\rightarrow 0$ is a short exact sequence of left comodules, then $0\rightarrow F(S_b)\rightarrow F(M)\rightarrow F(S_a)\rightarrow 0$ is a short exact sequence, and we use a standard argument to see that this induces a map $\Ff:\Ext^1(S_a,S_b)\rightarrow \Ext^1(F(S_a),F(S_b))$. Indeed, it is straightforward to see that this is well defined (up to equivalence of extensions); that $\Ff$ commutes with scalars follows easily by the linearity of $F$. Moreover, $\Ff$ is additive (on Baer sums of extensions); this follows as $F$ commutes with finite limits and colimits (since it is linear, so additive and commutes with finite products and coproducts, and it is left and right exact); and limits and colimits completely determine the vector space strucutre of $\Ext$. Finally, $\Ff$ is injective: this follows easily because $F$ preserves indecomposables: if the above short exact sequence $\zeta$ is not split (so $[\zeta]$ represents a nonzero element of $\Ext(S_a,S_b)$), then $M$ is indecomposable and so $F(M)$ is indecomposable; as a consequence, $F(\zeta)$ cannot be split.  Hence, there is a vector space embedding $\Ff:\Ext^1(S_a,S_b)\hookrightarrow \Ext^1(F(S_a),F(S_b))$, which shows that $\Ext^1(F(S_a),F(S_b))$ is infinite dimensional. Recall that a coalgebra is locally finite if and only if all $\Ext$ spaces between finite dimensional comodules are finite dimensional (for example, again by the results from \cite{I2} on characterizations of locally finite - \cite[Lemma 2.1]{I2} and the proof of it). Hence, $C$ cannot be locally finite, a contradiction.
\end{proof}

On the other hand, if a coalgebra $C$ is not locally finite, then there will be {\it some} infinite dimensional Ext space $\Ext(S,T)$ between simple modules $S,T$, and correspondingly, in the Ext quiver of $C$, there will be infinitely many arrows between two vertices. In case $S\not\cong T$, this will produce an embedding of the above quiver coalgebra $\Gamma_\infty(\KK)$ of the infinite Kronecker $\Gamma_\infty$ into $C$ (which will lie inside the first term of the coradical filtration); in the case when $S=T$ so $\Ext(S,S)$ is infinite dimensional for some simple $S$, we have that the quiver $I_\infty$ with one vertex and infinitely many loops embeds in the Ext quiver of $C$, and hence the coalgebra $(\KK I_\infty)_1$ - the first non-semisimple term of the coradical filtration of the quiver coalgebra $\KK I_\infty$ of $I_\infty$ - embedds into $C$. Thus, in this case, either $\KK \Gamma_\infty{\rm-comod}$ or $(\KK I_\infty)_1{\rm-comod}$ embeds into $C{\rm-comod}$ (a full faithful representation embedding). As noted above, if $D$ is any countable dimensional coalgebra, then $D$ embedds into the coalgebra $\KK_\NN=\lim\limits_{\stackrel{\longrightarrow}{n}}\KK\langle {\mathbf n}\rangle^0$, and so $D{\rm-comod}$ embeds into $\KK_\NN{\rm-comod}$ - a fully faithful representation embedding. Hence, in order to get embeddings between $D{\rm-comod}$ and $C{\rm-comod}$, one has to deal with the particular coalgebras $\KK \Gamma_\infty$, $(\KK I_\infty)_1$ and $\KK_\NN$.

We recall some terminology which is useful in constructing such embeddings by reinterpreting comodules over these coalgebras. Recall that if $Q$ is a quiver, the category ${\bf nrep}_Q$ of finite dimensional nilpotent representations of $Q$ are those finite dimensional representations which are annihilated by a monomial ideal of $\KK[Q]$ of finite codimension; the category ${\bf LocNilRep}_Q$ is the category of locally nilpotent representations of $Q$. By \cite{CKQ} (see also \cite{DIN1}), ${\bf nrep}_Q=\KK Q{\rm-comod}$ and ${\bf LocNilRep}_Q=\KK Q{\rm-Comod}$. Using this interpretation, we see that we have the following:

$\bullet$ $\KK \Gamma_\infty{\rm-comod}={\bf nrep}_{\Gamma_\infty}$ is the full subcategory of the category of finite dimensional representations of $\Gamma_\infty$, for which all but finitely many arrows of $\Gamma_\infty$, denoted by $y_n$, are  0;

$\bullet$ $(\KK I_\infty)_1{\rm-comod}$ is the full subcategory of the category of finite dimensional representations of $I_\infty$, in which the arrows (loops) of $I_\infty$, denoted $z_n$, act such that $z_iz_j=0$, and all but finitely many act as (i.e. are) $0$;

$\bullet$ $\KK_\NN{\rm-comod}$ is the full subcategory of the category of finite dimensional representations of $\KK\langle \NN\rangle=\KK\langle x_n| n\in \NN\rangle$ (and so of $I_\infty$) in which the all but finitely many of the $x_n$'s act as $0$.

We will construct two representation embeddings:

(1) {\it A full representation embedding of $\KK_\NN{\rm-comod}$ into $\KK \Gamma_\infty{\rm-comod}$}. If $M$ is a finite dimensional left $\KK_\NN$-comodule, define $F(M)=M \stackrel{\longrightarrow}{\dots} M$ to be the representation of $\Gamma_\infty$ having $M$ at the two vertices; the map corresponding to $y_{n+1}$ is defined to be the action of $x_n$ (a map from $M$ to $M$), and the map corresponding to $y_0$ is defined as identity $1_M:M\rightarrow M$. That is, the equations  $y_{n+1}=x_n$ and $y_0=1_M$ define the action of $\Gamma_\infty$. The definition on morphisms is similar: $F(f)=(f,f)$. This is easily seen to be faithful exact. Now if $(f,g):F(M)\rightarrow F(N)$ is a morphism of $\Gamma_\infty$ representations, then $g\circ y_0=y_0\circ f$ and since $y_0$ is identity, $f=g$; also, $f$ commutes with the action of $x_n=y_{n+1}$ which implies that $(f,g)=(f,f)=F(f)$ so $F$ is full. Since it is full and faithful, it preserves indecomposables and respects isomorphisms \cite{ASS}. 

(2) {\it A representation embedding of $\KK_\NN{\rm-comod}$ into $(\KK I_\infty)_1{\rm-comod}$}. If $M$ is a finite dimensional left $\KK_\NN$-comodule, define $G(M)=M\oplus M$ to be the representation of $I_\infty$ with $M\oplus M$ for the vertex, and the arrows defined as $z_{n+1}=\left(\begin{array}{cc} 0 & x_n \\ 0 & 0 \end{array}\right)$ and $z_0=\left(\begin{array}{cc} 0 & 1_M \\ 0 & 0 \end{array}\right)$ (these matrices are interpreted to act on $M\oplus M$ regarded as column vectors, in the usual way). Obviously, $z_iz_j=0$ and all but finitely many $z_n$'s act as $0$ since the same is true for the action of the $x_n$'s on $M$. Hence, by the observations above, $G(M)$ is a comodule over $(\KK I_\infty)_1$. The definition of $G$ on morphisms is the obvious one: for $f:M\rightarrow N$, it makes $G(f)$ act diagonally by $f$, that is, $G(f)=Diag(f)=\left(\begin{array}{cc} f & 0 \\ 0 & f \end{array}\right):\begin{array}{c} M \\ \oplus \\ M \end{array}\longrightarrow \begin{array}{c} N \\ \oplus \\ N \end{array}$. It is easy to see that $G(f)$ commutes with the action of all the $z_n$. Also, it is quite straightforward to note that $G$ is exact and faithful. While it is not full, it behaves ``almost" as such, and it is a representation embedding. First, we compute $\Hom(G(M),G(N))$. If $R=\left(\begin{array}{cc} a & b \\ c & d \end{array}\right)\in \Hom_{(\KK I_\infty)_1{\rm-comod}}(G(M),G(N))$, the condition that it commutes with $z_0$ easily implies that $R$ has the form $R=\left(\begin{array}{cc} a & b \\ 0 & a \end{array}\right)$, and using $Rz_n=z_nR$ for $n\geq 1$ we see that $a\in \End_{\KK\langle\NN\rangle}(M)$ commutes with the action of $x_n$ on $M$, and so $a$ is a morphism in  $\KK_\NN{\rm-comod}$. Now, if we have an isomorphism $G(M)\cong G(N)$, it has to be through an isomorphism $R$ of this form; moreover, since $R$ is bijective, it follows (by straight linear algebra) that $a:M\rightarrow N$ is bijective, and so $M\cong N$, showing that $G$ respects isomorphisms. Also, if $M$ is indecomposable and $R\in \End(G(M))$ is idempotent, using the shape of $R$ it follows immediately that $a=a^2$ in $\End(M)$, so $a=0$ or $a=1_M$. In both cases, an easy computation shows that $R^2=R$ only when $b=0$, so either $R=0$ or $R={\rm Id}_{M\oplus M}$, and thus, $G(M)$ is indecomposable. Hence, $G$ preserves indecomposables.




We summarize all the results of this section in the next statement.

\begin{theorem}\label{t.main2}
(A) If $C$ is a fully wild coalgebra, then for any countable dimensional coalgebra $D$, there is a fully faithful representation embedding of the category of $D$-comodules into $C{\rm-Comod}$. In particular, for any fully wild algebra $W$ (finite or infinite dimensional) and any countably generated algebra $A$, the category ${\rm lf-}A$ fully representation embedds in ${\rm lf-}W$ (respectively, into ${\rm Mod-}W$ if $W$ is finite dimensional).\\
(B) Let $C$ be a coalgebra. Then the following are equivalent: 
\begin{itemize}
\item[(i)] $C$ is not locally finite; 
\item[(ii)] There exists a representation embedding of either one of $\KK\Gamma_\infty{\rm-comod}$, $(\KK I_\infty)_1{\rm-comod}$ or $\KK_N{\rm-comod}$ into $C{\rm-comod}$.
\item[(iii)] for every countable dimensional (wild) coalgebra $D$, there is a representation embedding of $D{\rm-comod}$ (finite dimensional $D$-comodules) into $C{\rm-comod}$ (finite dimensional $C$-comodules). 
\end{itemize}
Moreover, if there are two different vertices of the Ext quiver of $C$ with infinitely many arrows between them, then such embeddings can be realized as full exact subcategories.
\end{theorem} 
\begin{proof}
(A) is just Theorem \ref{t.wildE}; (B) follows from Proposition \ref{p.KronekerEmbed}, and by putting together the previous remarks and representation embeddings.
\end{proof}

\section{Tame and wild coalgebras and localization}\label{s.tw}

Since the ``local" nature of the tame and wild properties is discussed, it is natural to study another type of meaning of the word ``local", namely, the behavior with respect to localization (of Grothendieck categories) of these two properties. In fact, in proving the tame-wild dichotomy for (one sided) semiperfect coalgebras, in \cite{si11} the fc-tame and fc-wild coalgebras are defined, which as we will see, are tightly related to localization. We begin by recalling some facts about localization on the Grothendieck categories $C{\rm-comod}$ of comodules.

\subsection*{Localization in comodules}

While localization of Grothendieck categories was introduced in Gabriel's foundational paper \cite{Ga}, this subject for comodule categories was initiated in \cite{NT1,NT2} and studied by many authors; see for example \cite{CGT}, \cite{GTNT}, \cite{JMN}, \cite{JMN2}. We list here the basic properties of localization (in the sense of Gabriel \cite{Ga}), for the case of locally finite category $C{\rm-comod}$. 

{\bf (1)} The Serre subcategories (i.e. full subcategories closed under subquotients and coproducts) of $C{\rm-comod}$ are of the form $H{\rm-comod}$ for a uniquely determined subcoalgebra $H$ of $C$ (this is a one-to-one correspondence between subcoalgebras and subcategories; \cite{DNR}). \\
{\bf (2)} The category $H{\rm-comod}$ is localizing (i.e. closed under extensions) if and only if $H$ is a coidempotent subcoalgebra of $C$, in the sense that $H\wedge H=H$. Moreover, $H$ is uniquely determined by the family $I_H=\{i\in I|\cf(S(i))\subseteq H\}$ ($H$ is obtained by ``closing" the subcoalgebra $\sum\{\cf(S(i))\}_{i\in I_H}$ under wedge), and there are one-to-one correspondences between subsets of $I$, coidempotent subcoalgebras of $C$, and such localizing subcategories.\\ {\bf (3)} Suppose for simplicity that $C$ is basic (which is no restriction up to Morita-Takeuchi equivalence), and fix a decomposition $C=\bigoplus\limits_{i}E(i)$ as before. For each localizing subcategory $H{\rm-comod}$ there is a unique idempotent $e$ of $C^*$ which is either $0$ or $\varepsilon$ on each $E(i)$ (namely, $e$ is $\varepsilon$ on $E(i)$ for $i\not\in I_H$, and $0$ on the other $E(i)$'s) such that the following diagram is commutative
$$\xymatrix{
C{\rm-Comod} \ar[rr]^T \ar[drr]_{(-)e} & & C{\rm-Comod}/H{\rm-Comod}\ar[d]^{\cong} \\
& & eCe{\rm-Comod} 
}$$
where $T$ is the localization functor, and $e(-)$ is the functor $C{\rm-Comod}\ni N\mapsto Ne\in {eCe}{\rm-Comod}$, and $eCe=\{\sum\limits_c e(c_1)c_2e(c_3)|c\in C\}$ is a coalgebra with (well defined) comultiplication $\Delta_{eCe}(ece)=\sum\limits_{c}ec_1e\otimes_\KK ec_2e$ (see \cite{CGT}, and also \cite{JMN, JMN2}), where in Sweedler notation $\Delta_C(c)=\sum\limits_{c}c_1\otimes c_2$. The counit of this coalgebra $eCe$ is the restriction of  $e\in C^*$ (and also of $\varepsilon$) to $eCe$ (in general, $eCe$ is not a subcoalgebra of $C$; it can be regarded as a quotient of $C$, but the quotient map does not respect counits). This is dual to (but also generalizes) localization in finite dimensional algebras. Hence, we denote simply $T=(-)e$ and will refer to this as the localization functor. We will refer to such an idempotent $e$, that corresponds to a quasifinite injective $E=\bigoplus\limits_{i\in I_H}E(i)$, as a {\it finite idempotent}.\\
{\bf (4)} The functor $T$ is naturally isomorphic to $Ce\square_C - $, with $Ce$ regarded naturally as a left-$eCe$-right-$C$-bicomodule ($\square$ is the cotensor product; see \cite{Tak}). Moreover, $T$ always has a right adjoint $S$, given by $S(-)=eC\square_{eCe} - $ (see \cite{Tak}; note that a left $D$-comodule $N$ over a coalgebra $D$ is equivalently a rational right module over the dual algebra $D^*$). Furthermore, we have $TS\cong{\rm Id}$ the corresponding identity functor, via the counit of the adjunction. \\
{\bf (5)} The closed (Serre) subcategory $H{\rm-Comod}$ is said to be colocalizing in case $T$ has a left adjoint $L$. By \cite[Proposition 1.10]{Tak} this is equivalent to $Ce$ being quasi-finite as a comodule over the coalgebra $eCe$, in which case there is an isomorphism of functors $L\cong {\rm coHom}_{eCe}(Ce,-)$, and we again have $TL={\rm Id}$, the corresponding identity functor, via the unit of the adjunction this time. Also, by \cite[Proposition 3.1]{NT2}, $H{\rm-Comod}$ is colocalizing if and only if $C/H$ is quasifinite as a left $C$-comodule, and in this case $H{\rm-Comod}$ is also localizing, and is closed under products in $C{\rm-comod}$. 

The following is a standard fact related adjoint pairs; parts of it are often used in considerations related to tame and wild finite dimensional algebras. We record it here for easy reference.

\begin{lemma}\label{l.adjoint}
Let $(F,G)=\xymatrix{\Aa\ar@/^/[r]^F & \Bb\ar@/^/[l]^G}$ be a pair of adjoint functors between abelian categories $\Aa,\Bb$ ($F$ being the left adjoint), such that $FG={\rm id}$ via the counit of the adjunction. Then $G$ is full and faithful, and hence, preserves indecomposables and respects isomorphisms. 
\end{lemma}
\begin{proof}
The equation $FG={\rm Id}$ easily implies $G$ faithful; $G$ is also full since $\Hom(G(A),G(A'))=\Hom(FG(A),A')=\Hom(A,A')$, and these identifications are made through the natural maps, and it is easy to see that identification between the first and last $\Hom$ is via $g\rightarrow G(g)$. 
\end{proof}

We need one more fact related to wild finite dimensional algebras. The following proposition follows essentially from the results of Drozd \cite{Dr} and a method presented in his proof of the tame-wild dichotomy. It says that if a there is a left exact additive functor $F:{\rm mod-}W\longrightarrow {\rm mod-}A$ which preserves indecomposables and reflects isomorphism types, where $W$ is a finite dimensional wild algebra, then one can find an (additive) functor $G:{\rm mod-}W'\longrightarrow {\rm mod-}A$, where $W'$ is a (possibly different) finite dimensional wild algebra, and such that $G$ preserves indecomposables, reflects isomorphism types and is furthermore exact. We refer the reader also to \cite[page 479]{CB}, where this replacement procedure is explained in detail; this is done by precomposing $F$ with a series of suitably chosen functors (this is also used implicitly in \cite{si11}).  One only needs to note that such a left exact additive functor $F$ is given by $F=-\otimes_W M $ for a bimodule $M$ which is finitely generated as a left $W$-module \cite{McL}. Hence, this proves the following Lemma, which is likely well known to specialists as is the above method of Drozd, but we could not find a direct reference, and we record it for further reference.

\begin{lemma}\label{l.weakwild}
Let $A$ be a finite dimensional algebra. Then the following are equivalent.\\
(i) $A$ is wild.\\
(ii) There is a left exact, additive functor $F:{\rm mod-}W\rightarrow {\rm mod-}A$ which preserves indecomposables and reflects isomorphisms, where $W$ is a wild algebra.\\
(iii) There is a right exact, additive functor $G:{\rm mod-}W\rightarrow {\rm mod-}A$ which preserves indecomposables and reflects isomorphisms, where $W$ is a wild algebra.
\end{lemma}
\begin{proof}
The equivalence between (i) and (ii) is explained above, and follows by the arguments (of Drozd \cite{Dr}) presented in \cite[page 479]{CB}. The equivalence of (ii) and (iii) follows by noting that if $D$ is the usual duality functor between left and right finite dimensional $W$-modules, then the functor $F:{\rm mod-}W\rightarrow {\rm mod-}A$ is left exact and has the properties required in (ii) if and only if the functor $DFD:{\rm mod-}W^{\rm op}\rightarrow {\rm mod-}A^{\rm op}$ is right exact and has these properties; finally, this ends the proof since the property of being wild is left-right symmetric, i.e. $A$ is wild if and only of $A^{\rm op}$ is wild.
\end{proof}

One may call such functors as in the remarks above, which are only half-exact but have all the properties of representation embeddings, {\it half-exact representation embeddings}.

To consider tame/wild vis-a-vis localization, we first deal with the finite dimensional case. The following is likely well-known.

\begin{proposition}\label{p.2}
Let $A$ be a finite dimensional algebra, and $e$ an idempotent of $A$. If $A$ is tame, then so is $eAe$. In particular, if $C$ is a finite dimensional tame coalgebra, then $eCe$ is tame for any idempotent $e\in A=C^*$.
\end{proposition}
\begin{proof}
We prove that if $eAe$ is wild, then $A$ is wild. Consider the localization $T:{\rm mod-}A\rightarrow {\rm mod-}eAe$ with right adjoint $T$ (and left adjoint $L$). Since $A$ is finite dimensional, these left and right adjoints take finite dimenisonal modules to finite dimensional modules, as it is easily observed. Then, apply Lemma \ref{l.adjoint} to observe that since $TS={\rm Id}$, $S$ is a left-exact representation embedding, and since $eAe$ is wild, 
by Lemma \ref{l.weakwild}, $A$ is wild. The part about coalgebras is equivalent to the one about algebras, since $(eCe)^*=eC^*e$ for any idempotent $e\in C^*$, and the coalgebra localization coincides with the algebra localization.
\end{proof}

We note that a direct proof (that does not appeal to the tame/wild dichotomy) of the above proposition is also possible, using the definition of tameness and the localization functor $T$. This can be done by using almost parametrizing families $L_1,\dots,L_n$ for $A$-modules to obtain almost parametrizing families $T(L_1),\dots,T(L_n)$ for $eAe$-modules; one needs to be careful to deal with dimension vectors. One way to do this is to consider parametrizations of dimension vectors which are less (in product ordering) then some fixed $\underline{d}$. Hence, one can show that all the $eAe$-modules of dimension vector $\leq\underline{d}$ can be obtained from $A$-modules of dimension vector less or equal to some $\underline{d'}$. The technical details are more tedious though, and are left to the interested reader.

We also need the following coalgebra observation, which is along the lines of the above mentioned approach of tame implies locally tame. 

\begin{proposition}
Let $C$ be an arbitrary coalgebra, $e$ an idempotent of $C^*$. \\
(i) If $N$ is a finite dimensional left $eCe$-comodule, then there exists a finite dimensional left $C$-comodule $M$ for which $Me=N$.\\
(ii) If $D$ is a finite dimensional subcoalgebra of $eCe$, then there exists a finite dimensional subcoalgebra $H$ of $C$ such that $D=eHe=fHf$, where $f=e\vert_H\in H^*$.
\end{proposition}
\begin{proof}
This is a straightforward application of the finiteness theorems for coalgebras and comodules. In fact, this can be also obtained as a consequence of general localization facts \cite{Ga}.\\
(i) First, since $N=TS(N)$, there is at least some $C$-comodule $X=S(N)$ with that property. If $y_1,\dots,y_n$ is a basis of $N=T(X)=Xe$, then there are $x_1,\dots,x_n\in X$ such that $x_ie=y_i$ ($Xe=\{xe|x\in X\}$). The subcomodule $M$ of $X$ generated by the $x_i$'s  is finite dimensional, and obviously, $Me=N$. \\
(ii) Similarly, take $x_i=ey_ie$ a basis of $D$; then the subcoalgebra $H$ of $C$ generated by the $y_i$'s is finite dimensional, and obviously has $eHe=D$. Finally, if $f=e\vert_H\in H^*$, then it is obvious that $fHf=eHe$, since, using Sweedler notation, $f\cdot h\cdot f=\sum\limits_hf(h_1)h_2f(h_3)=\sum\limits_he(h_1)h_2e(h_3)$ for $h\in H$ since $f(x)=e(x)$ for $x\in H$. 
 \end{proof}

We can now prove the main result of this section: that the ``tame" and ``not wild" properties are ``localizing", that is, that $C$ is tame/not wild if and only if every finite localization of $C$ is so. 

\begin{proposition}
Let $C$ be a pointed coalgebra (basic Schurian). Then: \\
(i) If $eCe$ is wild for some idempotent $e$, then $C$ is wild.\\
(ii) If $C$ is wild then there is a finite idempotent $e$ for which $eCe$ is wild.
\end{proposition}
\begin{proof}
(i) Assume $eCe$ is wild for some {\it arbitrary} idempotent. Since ``not-wild" is a local property (Lemma \ref{l.wildlocal}), this means that there is some finite dimensional subcoalgebra $D$ of $eCe$ which is wild. By the previous Proposition, $D=eHe$ for some finite dimensional subcoalgebra $H$ of $C$; but also $D=fHf$, with $f\in H^*$ an idempotent of $H^*$ ($f=e\vert_H$), which means that $D$ is a localization of $H$. Now, by proposition \ref{p.2}, we get that $H$ is wild, and finally, applying again the ``not wild is local" principle of Lemma \ref{l.wildlocal}, we get that $C$ is wild.\\
(ii) Suppose that $C$ is wild; then there is a finite dimensional subcoalgebra $H$ of $C$ which is wild. Let $e$ be a finite idempotent of $C^*$ such that $e\vert_H =\varepsilon_H$. This can be easily chosen as follows: since $C=\bigoplus\limits_{i\in I}E(i)=\bigcup\limits_{F{\rm\,finite}\subseteq I}\left(\bigoplus\limits_{i\in F}E(i)\right)$, we see that $H\subseteq \bigoplus\limits_{i\in F}E(i)$ for some finite subset $F$ of $I$, so $e$ can be chosen as $\varepsilon$ on $\bigoplus\limits_{i\in F}E(i)$ and $0$ on the complement $\bigoplus\limits_{i\notin F}E(i)$. Let $f=e\vert_H=\varepsilon\vert_H\in H^*$. \\
Then $f$ is the counit of $H$ and obviously $H=fHf\cong eHe$. (In fact, because of the choice of $e$, we see that {\it in this case}, $H=eHe$ can be regarded as a subcoalgebra of $C$.) Note that $eHe$ is a subcoalgebra of $eCe$. But $H\cong eHe$ is wild, and hence, $eCe$ is a wild coalgebra since it contains the wild subcoalgebra $eHe$.
\end{proof}

The above considerations justify the introduction of l-tame and l-wild coalgebras: $C$ is l-tame if every localization of $C$ at a finite idempotent $e$ is tame; and $C$ is l-wild if there is some localization $eCe$ at a finite idempotent $e$ which is wild (i.e. ``not l-wild" means every finite localization is ``not l-wild"). Now, using the tame/wild dichotomy for arbitrary coalgebras that we proved bevore, we can re-formulate this to the main result concerning tame/wild and localization, showing that l-tame and l-wild are not new notions but that the notions of tame and wild are equivalent, respectively, to their localizing correspondents, l-tame and l-wild, and so the dichotomy extends to the latter ones as well.

\begin{theorem}\label{t.main3}
A coalgebra $C$ (basic Schurian) is tame if and only if the localization $eCe$ at every every finite idempotent $e$ in $C^*$ is tame (equivalently, $C$ is l-tame); and $C$ is wild if and only if there is a finite idempotent $e$ for which the localization $eCe$ is wild (i.e. $C$ is l-wild). Consequently, any coalgebra is either l-tame (equivalently, tame) or l-wild (equivalently, wild), and not both.
\end{theorem}



\section{Relation to fc-tame and fc-wild}\label{s.fc}

In this section, we explain the l-tame and l-wild notions vis-a-vis Simson's fc-tame / fc-wild dichotomy. In \cite{si11}, in order to approach the tame-wild dichotomy, the notions of fc-tame and fc-wild were introduced. These notions used the category $C{\rm-Comod}^E_{\rm fc}$ of finitely $E$-copresented $C$-comodules, where $E=\bigoplus\limits_{i\in F}E(i)^{k_i}$ is a socle-finite $C$-comodule, that is, $soc(E)$ is finite dimensional (thus $F$ is a finite subset of $I$ and $k_i\in \NN$). The category $C{\rm-Comod}^E_{\rm fc}$ is the full subcategory of $C{\rm-Comod}$ consisting of all left $C$-comodules $X$ for which there is a short exact sequence $$0\rightarrow X\rightarrow E^k\rightarrow E^n$$
This category may obviously be defined for arbitrary $E$. Also, in \cite{si11} a comodule is said to be finitely copresented if there is an exact sequence $0\rightarrow X\rightarrow E_1\rightarrow E_2$, for $E_1,E_2$ socle-finite injective comodules. Equivalently, there exist some socle-finite injective $E$ for which $M$ is in $C{\rm-Comod}^E_{\rm fc}$. It is immediate to note also that $C{\rm-Comod}^E_{\rm fc}\subset C{\rm-Comod}^{E'}_{\rm fc}$  if $E\subset E'$ (so $E$ is a direct summand of $E'$). The category of {\it all} finitely copresented comodules is denoted by $C{\rm-Comod}_{\rm fc}$; that is, this is the full subcategory of $C{\rm-Comod}$ of all comodules which are finitely $E$-copresented for some socle-finite injective $E$ . A comodule $X$ which embeds in some $E^n$ for a socle-finite injective comodule $E$ is called {finitely cogenerated} in \cite{si11}.

There is another natural notion of finite cogeneration/copresentation which is also used sometimes, namely, finitely $C$-copresented comodules/$C$-cogenerated comodules. The category of finitely $C$-cogenetared comodules will be denoted $C{\rm-Comod}^C_{\rm fc}$ as above. This is a natural notion, since a comodule $M$ is finitely $C$-copresented if and only if $M^*$ is finitely presented (in the usual sense, as a module over $C^*$); and $M$ is finitely $C$-cogenerated if and only if $M^*$ is finitely generated; see also \cite{I5}. One sees that a finitely copresented comodule in the above sense of Simson is automatically finitely $C$-copresented.


A coalgebra $C$ is called Hom-computable or computable for short if $\Hom_C(E(i),E(j))$ is finite dimensional for all $i,j$. We note the following as a key fact about this situation; it implicitly plays a part in the work of \cite{si11}: if $C$ is computable, then for every finite idempotent $e$, the localization functor also has a left adjoint $L$. Indeed, this can be seen either directly since in this case, the coalgebra $eCe\cong \Hom_C(eC,eC)$ is finite dimensional, or by using that $eCe$ and $Ce$ are both finite dimensional, and therefore $Ce$ is quasi-finite over $eCe$ (in fact, $C$ is Hom-computable if and only if $eCe$ is finite dimensional for all finite idempotents $e$). Thus, one can apply a standard argument present in finite dimensional algebras to relate the categories $C{\rm-Comod}^E_{\rm fc}$ and $eCe{\rm-comod}={\rm mod-}(eCe)^*={\rm mod-}eC^*e$ via the functor $L$. Nevertheless, we note that such a connection exists in general, even in absence of the left adjoint $L$, via the right adjoint $S$ of $T$ (which always exists). Namely, the functor $S$ induces an equivalence of categories between $eCe{\rm-comod}$ (finite dimensional left $eCe$-comodules) and a subcategory of $C{\rm-comod}$ which is the ``image" of $S$. Let ${\rm im}(S)=S(eCe{\rm-comod})$ denote this subcategory. In order to give the most general and natural setting where the above correspondence works, but also for proper context inclusion, we need some terminology.

\subsection*{Coalgebras of finite type}

Recall \cite{HR} that a coalgebra is said to be of {\it finite type} if $C_1$ (the second term of the coradical filtration) is finite dimensional. It is proved in \cite[4.1.1]{HR} that such a coalgebra has the property that $C^*$ is (left and right) almost Noetherian (an algebra $A$ is left almost Noetherian if every cofinite left ideal of $A$ is finitely generated; a coalgebra $C$ for which $C^*$ is left almost Noetherian is said to be left {\it strongly reflexive} \cite{HR}). This is closely related to a few other notions, especially co-Noetherian and artinian comodules; \cite{GTNT, I4, I5}. We only note that if $C$ is strongly reflexive (in the above sense), then for every finite dimensional left subcomodule $N$ of $C$, $C/N$ is finitely $C$-cogenerated because $N^\perp$ is a finitely generated ideal (\cite[Section 2]{I4}; see also \cite{GTNT}). Moreover, similarly, if $N$ is a finite dimensional subcomodule of $C^n$, then $C^n/N$ is finitely $C$-cogenerated as well (since $(C^*)^n$ has the same property that submodules of finite codimension are finitely generated). This shows that in this case (when $C$ is strongly corelfexive, equivalently, $C^*$ is almost Noetherian), every finite dimensional $C$-comodule is finitely $C$-cogenerated. Also, when $C_0$ is finite dimensional then finitely cogenerated/copresented and finitely $C$-cogenerated/$C$-copresented comodules coincide. This shows that if $C$ is of finite type then every finite dimensional $C$-comodule is finitely copresented, and finitely copresented and finitely $C$-copresented comodules coincide.\\
A converse of this is also true. Assume $C_0$ is finite dimensional (so $C$ is socle-finite; in this case $C$ is said to be {\it almost connected}). 
If $C$ has the property that finite dimensional comodules are finitely ($C$-)copresented, then $C/C_0$ is finitely cogenerated: $C/C_0\hookrightarrow C^n$. We get that $C_1/C_0$ is finite dimensional, and since $C_0$ is finite dimensional, we see that $C_1$ is finite dimensional and so $C$ is of finite type. In fact, in view of the importance of the condition that finite dimensional $C$-comodules be finitely copresented, we introduce the following definition; recall that a comodule $M$ is quasifinite if $\Hom(S,M)$ is finite dimensional for every simple (equivalently, finite dimensional) comodule $S$.

\begin{definition}
(i) We say that a $C$-comodule $M$ over a coalgebra is f-quasifinite if $M/N$ is quasifinite for every finite dimenisonal $C$-subcomodule $N$ of $M$. \\
(ii) We say that a $C$-comodule $M$ over a coalgebra is f-finite if $M/N$ is socle-finite for every finite dimensional subcomodule $N$ of $M$. \\
(iii) We say that coalgebra $C$ is (left) f-quasifinite, respectively (left) f-finite, if $C$ is f-quasifinite, respectively, f-finite, as a left $C$-comodule.
\end{definition}

We note that the closely related notion of {\it strongly quasifinite} comodules (and coalgebras) also exists \cite{GTNT}: a comodule $M$ over a coalgebra is strongly quasifinite if $M/N$ is quasifinite for {\it every} subcomodule $N$ of $M$. Hence, ``f-quasifinite" is weaker than ``strongly quasifinite", and it is stronger than ``quasifinite". Also, the notion of f-finite is obviously related to the notion of finitely cogenerated (in the above sense of Simson \cite{si11}), and it appears in other places; for instance, if all injective indecomposables are f-finite then the category of $C$-comodules is closed under extensions inside the category of all $C^*$-modules \cite{I2}. Both locally finite and f-finite are notions which can be completely formulated in terms of the Ext quiver of $C$ (and of $C{\rm-Comod}$). While the notion of locally finite is equivalent to the Ext quiver having finitely many arrows between any two vertices, the notion of f-finite a stronger concept and is about the Ext quiver having vertices of finite degree, as seen below. The following proposition justifies its introduction vis-a-vis representation theoretic properties of $C$; its proof is not difficult, and since it is not essential for the next results, we leave it to the reader or future work. 

\begin{proposition}
(i) A coalgebra $C$ is left f-quasifinite if and only if every finite dimensional left $C$-comodule is finitely $C$-copresented; if $C$ is almost connected (i.e. $C_0$ is finite dimensional), this is further equivalent to the statement that every finite dimensional left $C$-comodule is finitely copresented.\\
(ii) A coalgebra is left f-finite if and only if every vertex in the Ext quiver of $C$ has finite in-bound degree (the number of arrows that arrive at each vertex is finite), and this is further equivalent to the statement that every finite dimensional left $C$-comodule is finitely copresented. In particular, in this case, $C$ is also locally finite, and f-quasifinite.
\end{proposition}

A graph theoretic characterization of f-quasifinite in terms of the Ext quiver is also possible, but lengthier to state (and not of immediate consequence here). As noted above, coalgebras (and also comodules) of finite type are such examples of f-finite and f-quasifinite coalgebras. Also, right semiperfect coalgebras are left f-finite (and hence, also left-quasifinite), and more generally, Hom-computable coalgebras (and comodules) are easily seen to be f-finite and f-quasifinite. Note that the inclusion of categories $C{\rm-Comod}_{\rm fc}\subseteq C{\rm-Comod}_{\rm fc}^C$ is an equality exactly when $C_0$ is finite dimensional ($C$ is almost connected). 

The following proposition shows the relevance of these notions in relating $C$-comodules with its localizations via the localization functor and its adjoint $S$. Recall that ${\rm im}(S)$ denotes the category $S(eCe{\rm-comod})$.

\begin{proposition}\label{p.equivalences}
Let $C$ be a coalgebra, $e\in C^*$ be an idempotent, and $E=eC$ the corresponding injective left $C$-comodule. Then:\\
(i) the (restriction of the) functor $S$ induces an equivalence of categories  $S:eCe{\rm-Comod}^{eCe}_{\rm fc}\longrightarrow C{\rm-Comod}^E_{\rm fc}$; moreover, when $e$ is finite, this produces an equivalence $eCe{\rm-Comod}_{\rm fc} \simeq C{\rm-Comod}^E_{\rm fc}$. \\
(ii) If the coalgebra $eCe$ is f-finite (in particular, if it is of finite type, in which case $e$ is a necessarily finite idempotent), then ${\rm im}(S)\subseteq C{\rm-Comod}^E_{\rm fc}$, \\
(iii) If $eCe$ is finite dimensional, ${\rm im}(S)=C{\rm-Comod}^E_{\rm fc}$, and $S$ gives an equivalence between $eCe{\rm-comod}$ and $C{\rm-Comod}^E_{\rm fc}$. In particular, this equality holds when $C$ is a Hom-computable coalgebra.
\end{proposition}
\begin{proof}
Since $S$ is full and faithful, one only needs to identify the image of $S$ on the corresponding subcategories of  $eCe{\rm-Comod}$.\\ 
(i) Since $S$ is left exact, and $S(eCe)=eC\square_{eCe}eCe=eC=E$, it follows easily that finitely $eCe$-copresented $eCe$-comodules go to finitely $E$-copresented $C$-comodules.  Conversely, if $0\rightarrow X\rightarrow E^n\rightarrow E^k$ is a finitely $E$-copresented left $C$-comodule, consider the diagram
$$\xymatrix{
0\ar[r] & X\ar[r]\ar@{..>}[d] & E^n \ar[r]\ar[d]^{\cong} & E^k\ar[d]^\cong \\
0\ar[r] & ST(X)\ar[r] & ST(E^n)\ar[r] & ST(E^k)
}$$
The solid vertical arrows are natural isomorphisms since $ST(E)=ST(eC)=eC\square_{eCe}(eC)e\cong eC=E$, and so it follows that the (induced) dotted arrow is an isomorphism. Moreover, $Y=T(X)$ is obviously finitely $eCe$-copresented (since $X$ is finitely $E$-copresented and $T$ is exact); hence, $X\cong S(Y)$ for a finitely $eCe$-copresented $eCe$-comodule $Y$; thus, the restriction and corestriction functor $S:eCe{\rm-Comod}^{eCe}_{\rm fc}\longrightarrow C{\rm-Comod}^E_{\rm fc}$ is full, faithful and dense, hence an equivalence. Also, when $e$ is finite, $eCe$ is almost connected and so $eCe{\rm-Comod}_{\rm fc}=eCe{\rm-Comod}^{eCe}_{\rm fc}$, and the last statement follows.\\
(ii) and (iii) When $eCe$ f-finite, we have 
$$eCe{\rm-comod}\subset eCe{\rm-Comod}_{\rm fc},$$
as noted in Proposition \ref{p.equivalences}(ii), so $S(eCe{\rm-comod})$ is contained in $C{\rm-Comod}^E_{\rm fc}$; when $eCe$ is finite dimensional, the statement follows as then we have equalities $eCe{\rm-Comod}_{\rm fc}=eCe{\rm-comod}=eCe{\rm-Comod}^{eCe}_{\rm fc}$. 
\end{proof}

Perhaps an important note at this point is that even though $S$ is not exact, it becomes exact as an equivalence between $C{\rm-Comod}^E_{\rm fc}$ and $eCe{\rm-Comod}^{eCe}_{\rm fc}$, where $C{\rm-Comod}^E_{\rm fc}$ is viewed as an exact abelian category with structure transported from $eCe{\rm-Comod}$; its structure as an abelian category is thus not enherited from $C{\rm-Comod}$ (cokernels might be different). 

Now, \cite[Lemma 2.7 and Proposition 2.8]{si11} and their proofs can be interpreted as statements connecting fc-tame/fc-wild with l-tame/and l-wild, and the next proposition is a re-formulation of those results (and follows from them): 

\begin{proposition}\label{p.Simson}
Let $C$ be a Hom-computable coalgebra. Then:\\
(i) $C$ is fc-tame if and only if $C$ is l-tame.\\
(ii) $C$ is fc-wild if and only if $C$ is l-wild.
\end{proposition}

Thus, via this proposition (which is a consequence of some of the work in \cite{si11}), the fc-tame/fc-wild dichotomy can be interpreted as a particular case of l-tame/l-wild dichotomy, which, in turn, is nothing else than the tame/wild dichotomy of coalgebras (in \cite{si11}, this was used to deduce the tame/wild dichotomy for computable coalgebras). It would be interesting to study these properties beyond Hom-computable coalgebras, at least in the case of coalgebras $C$ where $eCe$ is of finite type for all finite idempotents $e$. 

In fact, with the methods here, the part (ii) of the above can recovered directly. The following easy Lemma is useful for this purpose.

\begin{lemma}\label{l.extfc}
The category $C{\rm-Comod}^E_{\rm fc}$ is closed under extensions in $C{\rm-Comod}$.
\end{lemma}
\begin{proof}
This follows easily by the classical Horseshoe Lemma. If $0\rightarrow M'\rightarrow M\rightarrow M''\rightarrow 0$ is a short exact sequence of comodules with $M$ and $M'$ are in $C{\rm-Comod}^E_{\rm fc}$, consider the diagram
$$\xymatrix{
& 0 \ar[d] & 0\ar[d] & 0\ar[d] & \\
0\ar[r] & M' \ar[r]\ar[d] & M \ar[r]\ar@{..>}[d] & M''\ar[r]\ar[d] & 0 \\
0\ar[r] & E^k \ar[r]\ar[d] & E^{k}\oplus E^i \ar[r]\ar@{..>}[d] & E^i \ar[r]\ar[d] & 0 \\
0\ar[r] & E^l \ar[r] & E^{l}\oplus E^j \ar[r] & E^j \ar[r] & 0 \\
}$$
The first and last exact columns exist by hypothesis, and the middle (dotted) column can be completed as an exact sequence (and commutative squares) by the (appropriate version of the) Horseshoe Lemma.
\end{proof}

The next proposition recovers ``one-half" of the above mentioned result in Proposition \ref{p.Simson}. 

\begin{proposition}\label{p.fcwild}
(i) If $C$ is an f-finite coalgebra (in particular, if it is Hom-computable) which is wild, then $C$ is fc-wild.\\
(ii) If $C$ is a Hom-computable coalgebra which is fc-wild, then $C$ is wild.
\end{proposition}
\begin{proof}
(i) Obviously, a representation embedding ${\rm mod-}W\longrightarrow C{\rm-comod}$ with $W$ finite dimensional wild can be composed with the full embedding $C{\rm-comod}\hookrightarrow C{\rm-Comod}_{\rm fc}$ (which is true since $C$ is f-finite) to give a representation embedding ${\rm mod-}W\longrightarrow C{\rm-comod}_{\rm fc}$.\\
(ii) Let $F:{\rm mod-}W\longrightarrow C{\rm-Comod}_{\rm fc}$ with $W$-wild be a representation embedding. Let $S_1,\dots,S_n$ be representatives for the isomorphism types of simple (right) $W$-modules. Each $F(S_i)$ is finitely copresented, and since there are finitely many, there is some socle-finite injective such that $F(S_i)\in C{\rm-Comod}^E_{\rm fc}$. The previous lemma applied inductively shows that $F({\rm mod-}W)\subset C{\rm-Comod}^E_{\rm fc}$; but $C{\rm-Comod}^E_{\rm fc}$ is equivalent to $eCe{\rm-comod}$ by Proposition \ref{p.equivalences}(iii) for the corresponding finite idempotent $e$. This shows that $eCe$ is wild, and so $C$ is wild too by the results of the previous section. 
\end{proof}

\subsection{Fc-tameness}

We recall \cite{si11} that given a finitely copresented comodule $M$ and a minimal injective copresentation $0\rightarrow M\rightarrow E_0\rightarrow E_1$, then the vector $\cdn(M)=({\rm Soc}(E_0),{\rm Soc}(E_1))\in \ZZ^{(I)}\times \ZZ^{(I)}$ is called the coordinate vector of $M$. Fc-tameness is defined in terms of this: $C$ is fc-tame if finitely copresented comodules of $\cdn(M)=({\mathbf c}_0,{\mathbf c}_1)=v$ can be almost parametrized in the appropriate sense (we refer to \cite[Definition 2.2]{si11}). Namely, $C$ is fc-tame if for every bipartite vector $v=({\mathbf c}_0,{\mathbf c}_1)\in \ZZ^{(I)}\times \ZZ^{(I)}$, there is a {\it finitely copresented almost parametrizing family} of $C$-$\KK[T]$-bicomodules $L_i$ which are {\it finitely copresented bicomodules}, such that all but finitely many finitely copresented indecomposable $C$-comodules $M$ with $\cdn(M)=v$ are of the form $M\cong L_i\otimes_{\KK[T]} \KK[T]/(t-\lambda)$, $\lambda\in \KK$. Here, a finitely copresented bicomodule $L$ means a bicomodule for which there is an exact  $C$-$\KK[T]$-bicomodule sequence $0\rightarrow L\rightarrow E\otimes \KK[T]\rightarrow E'\otimes \KK[T]$ for some quasifinite injectives $E,E'$. 

We note that the condition can be formulated equivalently as follows: for all vectors $v\in \ZZ^{(I)}\times \ZZ^{(I)}$, there are $L_1,\dots,L_t$ which form a finitely copresented almost parametrizing family for all indecomposable comodules $M$ with $\cdn(M)\leq v$ (one needs only put all the appropriate families together). Consider now the localization functor $T=(-)e$, and assume $eCe$ is finite dimensional. Because of the equivalence of categories given by Proposition \ref{p.equivalences} (iii), it is not difficult to see that the family $L_1e,\dots,L_te$ becomes an almost parametrizing family for all $eCe$-comodules of $N$ coordinate vector $\cdn(N)\leq v_e$, where $v_e$ is obtained from $v=({\mathbf c}_0,{\mathbf c}_1)$ by deleting, in each of ${\mathbf c}_0$ and ${\mathbf c}_1$, all entries which correspond to simples $S$ for which $Se=0$ (hence, only the indices corresponding to simples remaining after localization are kept). Also, $L_ie$ remain finitely copresented as $eCe$-$\KK[D]$-bicomodules: an exact sequence $0\rightarrow L_i\rightarrow E^k\otimes \KK[T]\rightarrow E^l\otimes \KK[T]$ for $E=eC$ yields an exact sequence $0\rightarrow L_ie\rightarrow (eCe)^k\otimes \KK[T]\rightarrow (eCe)^l\otimes \KK[T]$. In particular, $L_ie$ are free as $\KK[T]$-modules. Finally, one can notice that given a dimension vector $\underline{d}$ over $eCe$, there is a coordonate vector $w=v_e$ for which all $eCe$-comodules $N$ of dimension vector $\dv(N)= \underline{d}$ have $\cdn(N)\leq w$ (simply because each such $eCe$-comodule $N$ is part of an exact sequence $0\rightarrow N\rightarrow (eCe)^n\rightarrow (eCe)^{n\dim(eCe)}$, where $n=|\underline{d}|=\dim(N)$). Hence, in particular, such indecomposable comodules admit almost parametrizing families (which will be subfamilies of the bicomodules $L_ie$, which are free over $\KK[T]$). 

In effect, this recovers the following, which is a result of \cite{si11}. 

\begin{proposition}
If $C$ is Hom-computable and fc-tame, then $eCe$ is an fc-tame, and hence a tame coalgebra (since $eCe$ is finite dimensional).
\end{proposition}

To go in the opposite direction, when $eCe$ is finite dimensional, one needs to use the left adjoint $H$ of the functor $T$, which is one of the important technical key facts in \cite{si11}. We note that in this case, $H(N)=N\otimes_{eC^*e}eC^*$ for any left $eCe$-comodule $N$. This is the case since it can be proved without difficulty that $N\otimes_{eC^*e}eC^*$ is rational as $C^*$-module; then, the two adjoints $H$ and $(-)\otimes_{eC^*e}eC^*$ of the functor $T$ on $eCe{\rm-Comod}$ must be naturally isomorphic. In fact, one essentially needs to observe that for a finite dimensional coalgebra $D$, fc-tame and tame are equivalent notions. 

Hence, in this line of thought, one can obtain that a Hom-computable $C$ is fc-tame if and only if it is l-tame, equivalently, it is tame; and since we also have that such a coalgebra $C$ is wild if and only if $C$ is fc-wild (Proposition \ref{p.fcwild}), using our direct proof of the tame/wild dichotomy for coalgebras, one obtains an alternative approach to the result of \cite{si11} that for a Hom-computable coalgebra $C$, tame is equivalent to fc-tame, wild is equivalent to fc-wild, any such coalgebra is either tame or wild but not both. 

We leave further details of such an alternative approach to fc-tame/fc-wild to the reader; nevertheless, we note that, as far as we can tell, some of the fine and more technical details required for the last part and present in \cite{si11} cannot be avoided.

\section{Questions}

\subsection{Connections to the the Brauer-Thrall 3 conjecture}\label{s.bt3}

We note here possible connections of the various embeddings present here between categories of locally finite modules to a conjecture due to D. Simson. In \cite{Si12}, the following is posed as a question.

\begin{conjecture}
If $A$ is a finite dimensional algebra which is not of finite type, then for any (infinite) cardinality $\lambda$, there is an indecomposable module $M$ of dimension $\dim(M)\geq \lambda$.
\end{conjecture}

As the classical two Brauer-Thrall conjectures asked whether algebras which are not of finite type have arbitrarily large finite dimensional indecomposable representations, we may call the above a Brauer-Thrall 3 (BT3) question. We remark that Ringel \cite{Ri1,Ri2} proved that the BT3 statement on existence of arbitrarily large indecomposables holds for the Kronecker quiver with two arrows, and for tame hereditary algebras. This means that it holds also for any algebra $A$ whose Ext quiver $Q$ is not Schurian, in the sense that $Q$ contains the Kroneker quiver $\Gamma_2$. One can see this because in this case the quiver coalgebra of $\Gamma_2$ embeds in $C=A^*$, and so there is an exact and full representation embedding of modules over $\Gamma_2$ into $C$-comodules (equivalently, $A$-modules); this embedding is then seen to ``preserve dimension", as we recall below. Hence, Ringel's work shows that the BT3 statement works for all such algebras.

The above conjecture is also proved to hold for several other classes of algebras in \cite{Si12}, such as fully wild algebras; this is a result of the fact that such algebras are Wild (this statement follows also from the more general embedding of Theorem \ref{t.wildE}). We recall here this method: if $B$ is fully wild, let $W$ be a wild algebra and a full faithful exact embedding $G:{\rm mod-}W\rightarrow {\rm mod-}B$, where $G$ can be assumed to be of the form $G(X)=P\otimes_W X$ for $P$ finitely generated projective over $W$. This $P$ is in fact finite dimensional and one can easily argue that $G$ preserves (infinite) dimension, and if indecomposable $W$-modules of arbitrary dimension exist, then the statement will hold for $B$. 

This can be done also by an argument independent of the finite dimensionality of $P$, which can potentially be used in other situations: if we assume that the indecomposable $B$-modules have bounded cardinality, then their isomorphism classes form a set $I$; since $W$ has modules of arbitrarily large cardinality, we can pick a set of non-isomorphic $W$ modules $J$ of cardinality larger than that of $I$. But since $G$ respects isomorphisms, the map $G:J\rightarrow I$, $X\longmapsto G(X)$ has to be injective, and so the cardinality of $I$ is at least as large as that of $J$, a contradiction. 

By the results in \cite[Theorem 3.1 and Corollary 3.2]{Si12}, there are finite dimensional wild algebras which satisfy the above conjecture (namely, the incidence algebra of any finite poset of wild representation type whose Tits form is not positive definite on vectors with entries non-negative integers), and hence such algebras $W$ can be used to show other algebras satisfy the conjecture provided suitable embeddings between large module categories can be found. In view of Lemma \ref{l.fininf} and the considerations immediately following it, and of Theorem \ref{t.wildE}, we ask the following question; if has has a positive answer, it would imply that any wild algebra is Wild, and hence, the above conjecture would hold for all wild algebras. We refer also to \cite{Sh} for connections between various other variations of the notion of wild.

\begin{question}
Let $F:C{\rm-Comod}\rightarrow D{\rm-Comod}$ be an exact functor which restricts and corestricts to a representation embedding $F_\vert:C{\rm-comod}\rightarrow D{\rm-comod}$. Does it follow that $F:C{\rm-Comod}\rightarrow D{\rm-Comod}$ is also a representation embedding?
\end{question}

Of course, to answer the above conjecture in the positive, the full positive answer to this question is not needed, but one needs to only show that {\it some} large indecomposables get preserved by such an embedding. 

In view of the embeddings between locally finite modules over countably generated algebras into that of modules over a finite dimensional fully wild algebra $A$, it seems natural to ask whether one can embed just ``any category" in ${\rm Mod-}A$.

\begin{question}
Let $C$ be an arbitrary coalgebra (or $C=\KK\langle x_i |i\in I\rangle^0$ the finite dual of the algebra of polynomials in some set of variables $I$). Can the category $C{\rm-Comod}$ be representation embedded into ${\rm Mod-}W$ for any fully wild algebra $W$?
\end{question}

\subsection{Further questions}

We end by listing a few other question which seem to naturally arise from this analisys. We formulate these here.

\begin{question}
Does the fc-tame/fc-wild dichotomy hold for arbitrary coalgebras? Does it at least hold for any interesting class properly containing Hom-computable coalgebras, such as f-finite, f-quasifinite, or strictly-quasifinite coalgebras? 
\end{question}

\begin{question}
Is ``tame" equivalent to ``fc-tame" in general for arbitrary coalgebras? Are they equivalent at least for f-finite, f-quasifinite, or for strictly-quasifinite coalgebras?
\end{question}

\begin{question}
The same question for wild: is ``wild" equivalent to ``fc-wild" for arbitrary coalgebras? Are they equivalent at least for f-finite, f-quasifinite, or for strictly-quasifinite coalgebras?
\end{question}

\begin{question}
Find general classes of coalgebras where the above stated questions have a positive answer.
\end{question}


In view of the results of Section \ref{s.we} - Theorem \ref{t.main2}, the following remains open.

\begin{question}
If $C$ is a locally finite (pointed) coalgebra, is there a representation embedding $C{\rm-comod}\longrightarrow {\rm mod-}\KK\langle z,w \rangle$? Note that it is enough to prove such an embedding exist for the quiver coalgebra of a locally finite quiver $Q$ (i.e. one for which only finitely many arrows exist between any two vertices). \\
Does such an embedding exist at least for the case when the vertices of $Q$ have finite (incoming and outgoing) degree? 
\end{question}


One should note at this point that, by the results of \cite[Sections 3,4]{HR}, a finitely generated algebra is always almost Noetherian (a ``Hilbert's basis theorem"), so if $A$ is such an algebra, its finite dual $A^0$ is left and right strongly reflexive, and hence, it is reflexive and locally finite (\cite{HR}). Thus, the coalgebra $\KK\langle x_1,\dots,x_n\rangle^0$ (the cofree coalgebra on a finite $n$-dimensional vector space; \cite{DNR}, see also \cite{AI}) is locally finite. Hence, since ${\rm mod-}\KK\langle z,w\rangle=\KK\langle z,w\rangle^0{\rm-comod}$, the above question can be rephrased and generalized as 

\begin{question}
(i) If $\Gamma,Q$ are locally finite quivers, when is there a representation embedding ${\bf nrep}_\Gamma=\KK \Gamma{\rm-comod}$ into ${\bf nrep}_Q=\KK Q{\rm-comod}$?  When is there such an embedding into ${\rm mod-}\KK\langle x_1,\dots,x_n \rangle$?\\
(ii) When is there a representation embedding ${\rm rep}_\Gamma={\rm mod-}\KK[\Gamma]$ into ${\rm rep}_Q={\rm mod-}\KK[Q]$? (here $\KK[\Gamma]$ and $\KK[Q]$ denote the path algebras).
\end{question}

In effect, one can use these embeddings to create a partial quasi-order between quivers ($\Gamma\preceq Q$ if ${\rm rep}_\Gamma$ representation embeds into ${\rm rep}_Q$), and an equivalence relation ($\Gamma$ and $Q$ are equivalent if $\Gamma\preceq Q$ and $Q\preceq \Gamma$), and one can talk about classifying quivers according to this equivalence (``embedding type").  We give below an answer to the last question in a particularly interesting case, the ``bounded" case. As noted, any pointed (i.e. basic Schurian) coalgebra $C$ can be embedded in the path coalgebra $\KK Q$ of its Ext quiver $Q$, and furthermore, $\KK Q{\rm-comod}$ embedds as a full abelian subcategory in ${\rm mod-}\KK[Q]$ (here $\KK[Q]$ is the path algebra of $Q$), which is the category of finite dimensional representations of $Q$ (see also \cite{din}). 

Let $C$ be a pointed coalgebra with Ext quiver $Q$ and assume that $\dim\Ext^{C,1}(L,T)$ when $L,T$ range over all simple left $C$-comodules, is bounded, so there is an $n_0$ such that $\dim\Ext^{C,1}(L,T)\leq n_0$ for all simple left comodules $L=S_i,T=S_j$. Equivalently, the number of arrows $n(i,j)$ between two vertices $i,j$ of $Q$ is bounded. This is equivalent to asking that for every simple left comodule $S_i$, the comodule $E(S_i)/S_i$ can be embedded into $C^{n_0}$. Equivalently, if $e_i$ is the corresponding idempotent of $C^*$ so that $e_i\vert_{E(S_i)}=\varepsilon$, then the (unique) maximal ideal $S_i^\perp$ of $E(S_i)^*$ is $n_0$-generated. In particular, it follows that simple comodules are finitely $C$-copresented, and by an inductive application of the Horseshoe Lemma (as in Lemma \ref{l.extfc}), it follows that all finite dimensional $C$-comodules are finitely $C$-copresented; this is close to the notion of $\Ff$-Noetherian coalgebra \cite{CNO} (a coalgebra $C$ is left $\Ff$-Noetherian if every closed cofinite ideal of $C^*$ is finitely generated), and f-quasifinite coalgebra. We note that such or closely related conditions were considered by many authors (see \cite{I2} and \cite[Theorem 4.8]{I2}, \cite{CNO, HR, T1, T2, Rad} and references therein). We have the following.

\begin{proposition}
Let $Q$ be a quiver with countably many vertices, such that the set $(n(i,j))_{i,j\in Q_0}$ is bounded. Then for every infinite field $\KK$, there is a representation embedding ${\rm mod-}\KK[Q]$ into ${\rm mod-}\KK\langle z,w\rangle$. In particular, if $C$ is a countable dimensional coalgebra whose Ext quiver has this property (i.e. the set $\{\dim(\Ext^{1,C}(S_i,S_j))\,|\, i,j\in I\}$ is bounded by some number $n$), then $C{\rm-comod}$ representation embeds in ${\rm mod-}\KK\langle z,w\rangle$ and hence, in ${\rm mod-}W$ for any wild algebra $W$.
\end{proposition}
\begin{proof}
Note that  it is enough to find a representation embedding from ${\rm mod-}\KK[Q]$ into ${\rm mod-}\KK\langle x_0,x_1,x_2,\dots,x_n\rangle$, since the latter can be representation embedded into ${\rm mod-}\KK\langle z,w\rangle$. For each pair of vertices $(a,b)$, enumerate the arrows $[a,b]_1, [a,b]_2, \dots, [a,b]_n$ and regard them as elements of $\KK[Q]$; in the case when there are only $k<n$ such arrows, we define the last $n-k$ such elements $[a,b]_i$ to equal $0\in\KK[Q]$. If $M=\bigoplus\limits_{a\in Q_0}M_a$ is a finite dimensional right $\KK[Q]$-module (so $M_a=0$ for all but finitely many vertices $a\in Q_0$), we let each $x_i$ to act on an $m_a\in M_a$ as the element $\sum\limits_{b\in Q_0}{[a,b]_i}$ (from the right). This is a formal sum (it can be regarded as an element of the algebra $(\KK Q)^*$, dual to the coalgebra $\KK Q$); however, the action is well defined, because $M$ is finite dimensional. Explicitly, $m_a\cdot x_i=(m_a \cdot [a,b]_i)_{b\in Q_0}\in\bigoplus\limits_{b\in Q_b}M_b$, and only finitely many of these components are non-zero. We {\it fix} a family of pairwise distinct elements $(\lambda_a)_{a\in Q_0}\subseteq \KK$. 
Finally, we let $x_0$ act as $m_a\cdot x_0=\lambda_a m_a$. This defines a right $\KK\langle x_0,x_1\dots,x_n\rangle$-module structure on $M$.\\
It is not difficult to see that if $\varphi=\bigoplus\limits_{a\in Q_0}\varphi_a:\bigoplus\limits_{a\in Q_0}M_a\rightarrow \bigoplus\limits_{a\in Q_0}N_a$ is a morphism of right $\KK[Q]$-modules, then $\varphi$  is a morphism of $\KK\langle x_0,x_1\dots,x_n\rangle$-modules too. Let us call the functor defined this way $F$. Obviously, $F$ is faithful. We note it is also full. If $\psi:M\rightarrow N$ is a morphism of $\KK\langle x_0,x_1\dots,x_n\rangle$-modules, proceed as in Proposition \ref{p.NN}, using the action of $x_0$ (where $M_a$ are all eigenvectors for different eigenvalues) to get that $\psi=\bigoplus\limits_{a\in Q_0}\psi_a$, with $\psi_a:M_a\rightarrow N_a$. By the condition $\psi(m_a\cdot x_i)=\psi(m_a)\cdot x_i$ for $m_a\in M_a$, one obtains equivalently $\psi_b(m_a\cdot [a,b]_i)=\psi_a(m_a)\cdot [a,b]_i$, and so $\psi$ is a morphism of representations. \\
It is easy to see that $F$ is also exact, and since $F$ is linear, full and faithful functor, it is such a desired representation embedding. 
\end{proof}

We note that the above proof can be used to find a concrete representation embedding from the category $C{\rm-comod}$ into ${\rm mod-}\KK\langle x_n|n\geq 0\rangle$ for a coalgebra of countable dimension. One proceeds as above with the Ext quiver $Q$ of $C$ which has at most countably many arrows between any two vertices. Of course, such ebmeddings are known to exist already by Theorem \ref{t.main2}, since the dual coalgebra $D=\KK\langle\NN\rangle^0$ of $\KK\langle\NN\rangle=\KK\langle x_n|n\geq 0\rangle$ is not locally finite, since its Ext quiver contains the quiver $Q_\infty$ with one vertex and countably many loops. We end by mentioning the following very natural (and quite interesting) question regarding representations of free algebras (and cofree coalgebras); we are not aware of it having been considered before.

\begin{question}
Determine the Ext quiver of the cofree coalgebra on an $n$ dimensional vector space $\KK\langle x_1,\dots,x_n\rangle^0$; equivalently, the Ext quiver of the category ${\rm mod-}\KK\langle x_1,\dots,x_n\rangle$, or the category ${\bf rep}_Q$ of the $n$-loop quiver.
\end{question}

It is to be expected that the category of locally finite modules over the non-commutative polynomial algebra is hereditary, in which case the answer to the previous question in effect would characterize the category of finite dimensional modules over the polynomial algebra, up to Morita equivalence. A more ambitious version of this question would be: describe completely the cofree coalgebra as a quiver with (co)relations, i.e. as a subcoalgebra of a quiver algebra up to Morita equivalence. We conjecture that this cofree coalgebra is hereditary, and hence its comodule category is completely described by the Ext quiver.

\end{document}